\documentclass[12pt,reqno]{amsart}
\usepackage{amsmath,amssymb,graphicx,mathrsfs,amsopn,stmaryrd,amsthm,amscd,bm,extarrows}
\usepackage{newtxtext}
\usepackage[dvipsnames]{xcolor}
\usepackage{ytableau}
\usepackage[colorlinks=true,citecolor=blue,linkcolor=blue]{hyperref}
\usepackage[english]{babel}
\usepackage[mathscr]{euscript}
\usepackage[left=2.5cm,
            right=2.5cm,
            top=2.85cm,
            bottom=2.75cm
            ]{geometry}

\usepackage[alphabetic,nobysame]{amsrefs}
\renewcommand{\MR}[1]{} 
\usepackage{collectbox}


\usepackage{dsfont}
\usepackage{tikz}    
\usetikzlibrary{arrows,matrix,decorations.markings,shapes.geometric}   
\usetikzlibrary{decorations.pathreplacing}

\usepackage{xy}
\allowdisplaybreaks
\xyoption{all}
\numberwithin{equation}{section}
\newtheorem{thm}{Theorem}[section]

\newtheorem{lem}[thm]{Lemma}

\theoremstyle{definition} 

\newtheorem{dfn}[thm]{Definition}
\theoremstyle{remark}
\newtheorem{rem}[thm]{Remark}

\newcommand{\beq}{\begin{equation}}
\newcommand{\eeq}{\end{equation}}
\newcommand{\be}{\begin{equation*}}
\newcommand{\ee}{\end{equation*}}

\newcommand{\bC}{\mathbb{C}}

\newcommand{\bZ}{\mathbb{Z}}
\newcommand{\bK}{\mathbb{K}}

\newcommand{\bN}{\mathbb{N}}
\newcommand{\I}{{\mathbb I}}
\newcommand{\iI}{{}^\imath\I}

\newcommand{\hf}{\tfrac12}
\newcommand{\arxiv}[1]{\href{http://arxiv.org/abs/#1}{\tt arXiv:\nolinkurl{#1}}}

\newcommand{\U}{\mathbf{U}}
\newcommand{\Ui}{\mathbf{U}^{\imath}}

\newcommand{\g}{\mathfrak{g}}

\newcommand{\id}{{\mathrm{id}}}

\newcommand{\pa}{\partial}

\newcommand{\gge}{\geqslant}
\newcommand{\lle}{\leqslant}
\newcommand{\la}{\lambda}

\newcommand{\ka}{\kappa}
\newcommand{\W}{{\mathbf W}}

\newcommand{\Z}{{\mathbf Z}}

\newcommand{\Sym}{{\mathrm{Sym}}}

\newcommand{\aTheta}{{\acute{\Theta}}}

\newcommand{\bv}{\mathbf{v}}
\newcommand{\bw}{\mathbf{w}}

\newcommand{\brown}[1]{{\color{brown}#1}}

\begin{document}
\pagestyle{myheadings}
\setcounter{page}{1}

\title[GKLO-type representations of shifted affine $\mathrm{i}$quantum groups]{Shifted affine $\mathrm{i}$quantum groups of quasi-split ADE types}

\author{Kang Lu}
\address{Shenzhen International Center for Mathematics and Department of Mathematics, Southern University of Science and Technology, Shenzhen, China}\email{kang.lu.math@gmail.com, luk@sustech.edu.cn}
\author{Weiqiang Wang}
\address{Department of Mathematics, University of Virginia, 
Charlottesville, VA 22903, USA}\email{ww9c@virginia.edu}
\author{Alex Weekes}
\address{Département de mathématiques, Université de  Sherbrooke, Sherbrooke, QC, Canada}
\email{alex.weekes@usherbrooke.ca}

\subjclass[2020]{Primary 17B37.}
\keywords{Drinfeld presentation, affine iquantum groups}

\begin{abstract}
We formulate shifted affine iquantum groups of arbitrary quasi-split ADE types via Drinfeld presentations. We construct GKLO-type representations of shifted affine iquantum groups via algebras of difference operators, which allow us to construct truncated shifted affine iquantum groups. This provides a q-deformation of truncated shifted iYangians in our prior work arising as a quantization of affine Grassmannian islices. 
\end{abstract}
	
\maketitle

\thispagestyle{empty}

\section{Introduction}

The interaction between geometric representation theory and quantum algebras has been fruitful and beneficial in both directions. 
The GKLO representations of shifted Yangians have played an indispensable role in bridging the connections between affine Grassmannian slices, Coulomb branches, and shifted Yangians; see \cite{GKLO, KWWY14, BFN2}. Moreover, their $K$-theoretic counterparts lead naturally to multiplicative analogues, namely multiplicative affine Grassmannian slices, $K$-theoretic Coulomb branches, and shifted quantum affine algebras; see \cite{FT19,FT19b}. In both contexts, GKLO-type representations \cite{GKLO, GKLO05b} arise as explicit realizations of these algebras by (additive or multiplicative) difference operators, reflecting the underlying integrable systems and providing a concrete coordinate function realization for the quantum algebras.

More recently, this picture has been extended to a general $\imath$-setting \cite{LWW25iSlcies, LWW25iY}. In this framework, one replaces the usual affine Grassmannian slices by their fixed-point loci under involutions, known as affine Grassmannian \emph{islices}. The latter are quantized by shifted iYangians (or twisted Yangians in Drinfeld presentations) associated to arbitrary quasi-split ADE symmetric pairs, which admit so-called iGKLO representations incorporating the additional symmetry imposed by the involution. This development has revealed a parallel and highly structured $\imath$-generalization, in which geometry, algebra, and explicit operator realizations continue to align in a strikingly coherent manner. We mention that closely related constructions with varied generalities for type AI or AIII have also been carried out by several groups in \cite{SSX25, BPT25, Wan26} with different perspectives.

The purpose of the present work is to introduce a $q$-deformation of the constructions in \cite{LWW25iY}. More precisely, we formulate and investigate \emph{shifted affine iquantum groups of arbitrary quasi-split ADE type}, and construct for them a family of iGKLO-type representations by means of algebras of $q$-difference operators. 

Let $\g$ be a simple Lie algebra of type ADE. Let $\tau$ be an involution  of its Dynkin diagram of $\g$ (with $\tau=\id$ allowed), leading to a quasi-split Satake diagram $(\I,\tau)$. Associated to this data is a (universal) affine quantum symmetric pair $(\U,\Ui)$ from \cite{LW22} (different from earlier versions of Letzter and Kolb), where $\Ui$ contains generators $\bK_i$, for all $i\in \I$. The Drinfeld-type current presentations for these affine iquantum groups $\Ui$ have been established in recent years \cite{LW21, LWZ24, LPWZ25}. In this paper, we introduce shifted affine iquantum groups, denoted $\Ui_\mu$, depending on an integral coweight $\mu$, by suitably modifying the Drinfeld presentations of $\Ui$.

The algebra $\Ui_\mu$ is generated by $\bK_i^{\pm1}$ and current-type elements in generating function forms
$B_i(u)$ and $\aTheta_i(u),$
subject to relations that simultaneously generalize:
\begin{itemize}
    \item the Drinfeld presentation of affine quantum and iquantum groups,
    \item and the shift/truncation conditions familiar from shifted Yangians and iYangians.
\end{itemize}
A key feature is the presence of the diagram involution $\tau$, which couples the generators associated to $i$ and $\tau i$, leading to modified commutation and Serre-type relations. The shift parameter $\mu$ leads to various truncations, which are expected to correspond to \emph{multiplicative} affine Grassmannian islices, finite-dimensional varieties arising via the multiplicative version of \cite{LWW25iSlcies}.


The main result of the paper is the construction of \emph{iGKLO representations} of the shifted affine iquantum groups. Concretely, for a pair of coweights $\mu \lle \lambda$ with $\la$ dominant, we construct an explicit algebra homomorphism (see Theorem \ref{thmq})
\[
\Phi^\lambda_\mu : \mathbf{U}^\imath_\mu \longrightarrow \mathscr{A}^q.
\]
Here the ``quantum torus" $\mathscr{A}^q$ is a natural $q$-analogue of the algebra introduced in \cite{LWW25iY}. More precisely, $\mathscr{A}^q$  is an algebra generated by the multiplicative ``coordinate" variables $w_{i,r}$ and $q$-difference operators $\partial_{i,r}$ (which act by $w_{i,r} \mapsto q^2 w_{i,r}$)
subject to natural commutation relations. The image of $B_i(u)$ under $\Phi^\lambda_\mu$ takes the form of a sum over indices $r$, with coefficients involving:
\begin{itemize}
    \item delta functions $\delta(\frac{w_{i,r}}{qu})$,
    \item rational functions encoding interactions between nodes of the Satake diagram,
    \item and the $q$-difference operators $\partial_{i,r}^{\pm1}$.
\end{itemize}

The iGKLO representations provide a concrete and computable realization of the algebra $\Ui_\mu$, and can be viewed as a multiplicative analogue of the iGKLO representations constructed previously for shifted iYangians \cite{LWW25iY}. Note that some of the series $B_i(u)$ here contain an inhomogeneous constant term (this feature already appeared in the setting of shifted iYangians \cite{LWW25iSlcies, LWW25iY} which was demanded by its geometric applications). In the case of symmetric pairs of diagonal type, our constructions specialize to (variants of) shifted affine quantum groups and their GKLO representations in \cite{FT19}. 

The proof of the main Theorem   \ref{thmq} proceeds by verifying that the explicit formulas preserve all defining relations of the shifted affine iquantum group, including the most intricate Serre-type relations. The non-split case requires additional technical arguments beyond the split case.


One of the main motivations for this work comes from geometric representation theory. In our earlier work \cite{LWW25iSlcies}, shifted twisted Yangians were shown to arise as quantizations of affine Grassmannian islices. The present construction enables the definition of \emph{truncated shifted affine $\imath$-quantum groups}, which are expected to quantize the top-dimensional components of multiplicative affine Grassmannian islices, and to relate to $K$-theoretic iCoulomb branches. We hope that the formulas offered in this paper will facilitate making such  connections precise.


The paper is organized as follows.
In Section \ref{sec:shiftediQG}, we introduce shifted affine iquantum groups using a variant of the Drinfeld presentation.
In Section \ref{sec:iGKLO}, we define the algebra of $q$-difference operators and construct the iGKLO representations.
Sections \ref{sec:pf-split}--\ref{sec:pf-quasi-split} are devoted to the proof of the main Theorem \ref{thmq}, treating the split and non-split cases separately. 

This paper expands on our unpublished notes written in early 2024 and was announced in \cite{LWW25iSlcies, LWW25iY}. A version of the GKLO-type construction for shifted affine iquantum groups of {\em split} ADE type has also appeared in \cite{LP26}.

\vspace{2mm}
\noindent {\bf Acknowledgement.}
W.W. is partially supported by the NSF grant DMS-2401351. He thanks Academia Sinica and the NSTC in Taipei as well as the University of Hong Kong for providing a stimulating environment and support during his visits. A.W.~is supported by an NSERC Discovery Grant. 

\section{Shifted affine iquantum groups of quasi-split ADE type}
\label{sec:shiftediQG}

In this section, we formulate the notion of shifted affine iquantum groups of arbitrary quasi-split ADE type. 

\subsection{Affine iquantum groups}

Let $C=(c_{ij})_{i,j\in \I}$ be the Cartan matrix of type ADE, and let $\g$ be the corresponding simple Lie algebra. We fix a simple system $\{\alpha_i\mid i\in \I\}$ with corresponding set $\Delta^+$ of positive roots. Let $\tau$ be an involution of the  Dynkin diagram of $\g$, i.e., $c_{ij}=c_{\tau i,\tau j}$ such that $\tau^2=\mathrm{id}$; note that $\tau=\id$ is allowed. We refer to $(\I,\tau)$ as a quasi-split Satake diagram and call the Satake diagram split if $\tau=\id$. The split Satake diagrams formally look identical to Dynkin diagrams.
Denote $\I_0$ the set of fixed points of $\tau$ in $\I$, i.e., $\I_0 =\{i\in \I \mid \tau i=i\}$. Let $\I_1$ be a set of representatives for $\tau$-orbits in $\I$ of length 2 and define $\I_{-1}=\tau \I_1$; $\I_1$ can be conveniently chosen so that its underlying Dynkin subdiagram is connected. Then $\I =\I_1 \sqcup \I_0 \sqcup \I_{-1}$. Set \begin{align}  \label{eq:iI}
\iI=\I_1 \sqcup \I_0.
\end{align}

We first recall the Drinfeld presentations of affine iquantum groups $\Ui$ of quasi-split ADE type, following the constructions from \cite{LW21,LWZ24,LPWZ25}; the $\Ui$ used here is the universal version from \cite{LW22} (different from Letzter-Kolb), which contains generators $\bK_i$ for all $i\in \I$. Different relations depending on 
Cartan integers are required as there are several non-isomorphic subalgebras of $\Ui$ of affine rank one and two types. 

For $n\in \bZ$, let $[n]=\frac{q^{n}-q^{-n}}{q-q^{-1}}$. Introduce the generating series
\begin{align*}
B_i(u)=\sum_{r\in\bZ}B_{i,r}u^r,\quad \Theta_i(u)=1+\sum_{r>0}(q-q^{-1})\Theta_{i,r}u^r,\quad \delta(u)=\sum_{r\in \bZ} u^r,
\end{align*}
and the following notation
\begin{align*}
&\mathbb S_{i,j}(u_1,u_2|v)=\mathrm{Sym}_{u_1,u_2}\big[B_i(u_1),[B_i(u_2),B_j(v)]_q\big]_{q^{-1}}\\
&=\Sym_{u_1,u_2}\big(B_j(v)B_{i}(u_1)B_{i}(u_2)-[2]B_i(u_1)B_j(v)B_{i}(u_2)+B_{i}(u_1)B_{i}(u_2)B_j(v)\big).
\end{align*}
Here and below we denote $[X,Y]_v=XY-vYX$. 
We will frequently use the identities
\beq\label{delta}
f(u)\delta\Big(\frac{a}{u}\Big)=f(a)\delta\Big(\frac{a}{u}\Big),\qquad G(u,v)\delta\Big(\frac{a}{u}\Big)\delta\Big(\frac{b}{v}\Big)=G(a,b)\delta\Big(\frac{a}{u}\Big)\delta\Big(\frac{b}{v}\Big).
\eeq
\begin{dfn}
The quasi-split affine iquantum group $\Ui$ is generated by the elements $B_{i,r},\Theta_{i,s},\bK_{i}^{\pm 1}$, where $i\in \I$, $r\in\bZ$, $s\in \bZ_{>0}$, subject to the following relations, for $i,j\in \I$:
\begin{align}
&[\bK_i,\bK_j]=[\bK_i,\Theta_j(u)]=[\Theta_i(u),\Theta_j(v)]=0,\label{HH}\\
&\bK_iB_j(u)=q^{c_{\tau i,j}-c_{ij}}B_j(u)\bK_i,\label{HB0}\\
&\Theta_i(u)B_j(v)=\frac{1-q^{-c_{ij}}uv^{-1}}{1-q^{c_{ij}}uv^{-1}}\cdot \frac{1-q^{c_{\tau i,j}}uv }{1-q^{-c_{\tau i,j}}uv}B_j(v)\Theta_i(u),\label{HB}\\
&\big[B_i(u),B_{\tau i}(v)\big]=\frac{\delta(uv)}{q-q^{-1}}(\bK_{\tau i}\Theta_i(u)-\bK_i\Theta_{\tau i}(v)),\quad \text{ if }c_{i,\tau i}=0, \label{BBitaui0}\\
&(q^2u-v)B_i(u)B_i(v)+(q^2v-u)B_i(v)B_i(u)\notag\\&\qquad=q^{-2}\frac{\delta(uv)}{q-q^{-1}}\big((q^2u-v)\Theta_i(v)+(q^2v-u)\Theta_i(u)\big)\bK_i,\quad \text{if }i=\tau i,\label{BBi=taui0}\\
&(q^{-1}u-v)B_i(u)B_{\tau i}(v)+(q^{-1}v-u)B_{\tau i}(v)B_i(u)\notag\\&\qquad=\frac{\delta(uv)}{1-q^{2}}\big((u-qv)\bK_i\Theta_{\tau i}(v)+( v-qu)\bK_{\tau i}\Theta_i(u)\big),\quad \text{if }c_{\tau i,i}=-1,\label{BBitaui-1}\\
&[B_i(u),B_j(v)]=0,\qquad \text{if }c_{ij}=0, j\ne \tau i,\\
&(q^{c_{ij}}u-v)B_i(u)B_j(v)+(q^{c_{ij}}v-u)B_j(v)B_i(u)=0,\quad\text{if }j\ne \tau i,
\end{align}
and the Serre relations: for $c_{ij}=-1$,$j\neq \tau i\neq i$,
\[
    \mathbb S_{i,j}(u_1,u_2|v) =0,
   \]
and  for $c_{ij}=-1$, $i=\tau i$,
\beq
\begin{split}\label{SSGF2}
\mathbb S_{i,j}(u_1,u_2|v)
&= -\frac{\delta(u_1u_2)}{q-q^{-1}}  \Sym_{u_1,u_2} \left(\frac{[2] v u_1^{-1} }{1 -q^{2}u_2u_1^{-1}}[\Theta_i(u_2), B_j(v)]_{q^{-2}} \right.
 \\ 
&\qquad\qquad\qquad\qquad\qquad
+ \left.\frac{1 +u_2u_1^{-1}}{1 -q^{2}u_2u_1^{-1}}[B_j(v),\Theta_i(u_2)]_{q^{-2}} \right)\bK_i, 
\end{split}
\eeq
and  for $c_{i,\tau i}=-1$,
\beq
\begin{split}\label{SSGF3}
	& \mathbb{S}_{i,\tau i}(u_1,u_2|v)= \frac{1}{q-q^{-1}}\Big(-q^{-1}[2]\Sym_{u_1,u_2}\delta(u_2 v) \frac{1-q u_2^{-1} v}{1-q^{-2} u_1 u_2^{-1}} B_i(u_1)\Theta_{\tau i}(v)\bK_i
			\\ 
	&~~\qquad\qquad\qquad\qquad \quad  +[2]\Sym_{u_1,u_2}\delta(u_2 v) \frac{1-qu_2^{-1}v}{1-q^2 u_1 u_2^{-1}}\Theta_{\tau i}(v)\bK_i B_i(u_1)
			\\ 
	&~~\qquad\qquad\qquad\qquad\quad +q[2]\Sym_{u_1,u_2}\delta(u_2 v)\frac{u_1^{-1}v-q u_1^{-1} u_2}{1-q^2 u_1^{-1} u_2} B_i(u_1)\Theta_{i}(u_2)\bK_{\tau i}
			\\
	&~~\qquad\qquad\qquad\qquad\quad+q^{-2}[2]\Sym_{u_1,u_2}\delta(u_2 v) \frac{q u_1^{-1}u_2 - u_1^{-1}v}{1-q^{-2}u_1^{-1}u_2} \Theta_{i}(u_2)\bK_{\tau i} B_i(u_1)\Big). 
\end{split}
\eeq
\end{dfn}
\begin{rem}
Note that for simplicity we do not include the central generators $C^{\pm 1}$. To recover the version with $C^{\pm 1}$, it suffices to do the substitution $u\to u C^{\frac12}$ for the spectral parameter; see e.g. \cite[\S4.3]{Lu24iso}.
\end{rem}

\subsection{A variant of Drinfeld presentation}
For our purpose, it is convenient to work with a somewhat different presentation, using another set of imaginary root vectors defined by the rule (cf. \cite[Lemma 2.9]{LW21}):
\beq\label{NewCartan}
\acute{\Theta}_i(u)=\sum_{r\gge 0}\acute{\Theta}_{i,r}u^{r}:=\frac{1-q^{-c_{\tau i,i}}u^2}{1-u^2}\bK_{\tau i}\Theta_i(u).
\eeq
The following lemma indicates that certain relations can be written in a simpler form.

\begin{lem}
Assuming the validity of the relations \eqref{HH}--\eqref{HB}, the following equivalences hold.
\begin{enumerate}
	\item The relation \eqref{BBi=taui0} is equivalent to 
\beq\label{BBi=taui}
(q^2u-v)B_i(u)B_i(v)+(q^2v-u)B_i(v)B_i(u)=\frac{\delta(uv)(u-v)}{q^{-1}-q}(\acute{\Theta}_i(u)-\acute{\Theta}_i(v)).
\eeq
	\item The relation \eqref{BBitaui-1} is equivalent to 
\beq
(q^{-1}u-v)B_i(u)B_{\tau i}(v)+(q^{-1}v-u)B_{\tau i}(v)B_i(u)=\frac{\delta(uv)(u-v)}{q^{2}-1}(\acute{\Theta}_i(u)-\acute{\Theta}_{\tau i}(v)).
\eeq
	\item The Serre relation \eqref{SSGF2} is equivalent to 
\beq\label{NewSerre1}
\mathbb S_{i,j}(u_1,u_2|v)=\frac{\delta(u_1u_2)(u_1-u_2)v}{(qu_1-v)(qu_2-v)}(\acute{\Theta}_i(u_2)-\acute{\Theta}_{i}(u_1))B_j(v).
\eeq
	\item The Serre relation \eqref{SSGF3} is equivalent to 
\beq\label{NewSerre2}
\begin{split}
\mathbb S_{i,j}(u_1,u_2|v)=\Sym_{u_1,u_2}\frac{\delta(u_2v)(1+q^2)(1-v^2)}{(1-qu_1^{-1}v)(1-q^2u_1v)}\big(\acute{\Theta}_i(u_2)-\acute{\Theta}_{\tau i}(v)\big)B_i(u_1).
\end{split}
\eeq
\end{enumerate}
\end{lem}
\begin{proof}
The arguments for these relations via the identities \eqref{delta} are all similar. Hence we only prove the most complicated relations \eqref{NewSerre1}--\eqref{NewSerre2}, skipping the others. 

We start with \eqref{NewSerre2}. We substitute \eqref{NewCartan} into
\[
-q^{-1}\delta(u_2 v) \frac{1-q u_2^{-1} v}{1-q^{-2} u_1 u_2^{-1}} B_i(u_1)\Theta_{\tau i}(v)\bK_i+\delta(u_2 v) \frac{1-qu_2^{-1}v}{1-q^2 u_1 u_2^{-1}}\Theta_{\tau i}(v)\bK_i B_i(u_1)
\]
and use \eqref{HB0}--\eqref{HB} to move $B_i(u_1)$ to the right. Since $\delta(u_2v)$ is in each term, by \eqref{delta} we are free to interchange $v$ and $u_2^{-1}$. Then we obtain
\begin{align*}
-\delta(u_2v)\frac{q^2(1-v^2)(1-q^{-1}u_1^{-1}v)}{(1-qu_1^{-1}v)(1-q^2u_1v)}\acute{\Theta}_{\tau i}(v)B_i(u_1)	+\delta(u_2v)\frac{1-v^2}{1-q^2u_1u_2^{-1}}\acute{\Theta}_{\tau i}(v)B_i(u_1),
\end{align*}
which simplifies to
\beq\label{helper01}
\frac{\delta(u_2v)(1-v^2)(1-q^2)}{(1-qu_1^{-1}v)(1-q^2u_1v)}\acute{\Theta}_{\tau i}(v)B_i(u_1).
\eeq
A similar calculation shows that
\[
q\delta(u_2 v)\frac{u_1^{-1}v-q u_1^{-1} u_2}{1-q^2 u_1^{-1} u_2} B_i(u_1)\Theta_{i}(u_2)\bK_{\tau i}+q^{-2}\delta(u_2 v) \frac{q u_1^{-1}u_2 - u_1^{-1}v}{1-q^{-2}u_1^{-1}u_2} \Theta_{i}(u_2)\bK_{\tau i} B_i(u_1)
\]
simplifies to
\beq\label{helper02}
\frac{\delta(u_2v)(v^2-1)(1-q^2)}{(1-qu_1^{-1}v)(1-q^2u_1v)}\acute{\Theta}_{i}(u_2)B_i(u_1).
\eeq
Now combining \eqref{helper01}--\eqref{helper02} and multiplying with the common factor $[2](q-q^{-1})^{-1}$, one obtains \eqref{NewSerre2}. 

The relation \eqref{NewSerre1} is obtained by applying the same argument to \cite[Lemma 4.2]{SW26}, which is the equivalent form of \eqref{SSGF2}.
\end{proof}

\subsection{Shifted affine iquantum groups}
Now we introduce shifted affine iquantum groups of quasi-split ADE type with the help of new variants of relations \eqref{BBi=taui}--\eqref{NewSerre2}.
Let $\mu$ be an integral coweight and set $\ell_i=\langle \mu,\alpha_i\rangle$.
\begin{dfn}
The shifted affine iquantum group $\Ui_\mu$ of quasi-split type is the $\bC(q)$-algebra generated by the elements $\Theta_{i,r}$, $\bK_{i}^{\pm 1}$, and $B_{i,s}$, where $i\in \I$, $r>-\ell_{\tau i}$, $s\in \bZ$, subject to the following relations, for $i,j\in \I$:
\begin{align}
&[\aTheta_i(u),\aTheta_j(v)]=0,\label{rel:HH}\\
&\aTheta_i(u)B_j(v)=\frac{(q^{c_{ij}}u-v)(q^{c_{\tau i,j}}u^{-1}-v)}{(u-q^{c_{ij}}v)(u^{-1}-q^{c_{\tau i,j}}v)}B_j(v)\aTheta_i(u),\label{rel:HB}\\
&[B_i(u),B_{\tau i}(v)]=\frac{\delta(uv)}{q-q^{-1}}\big(\aTheta_i(u)-\aTheta_{\tau i}(v)\big), \quad \text{ if }c_{i,\tau i}=0,\label{rel:BB1}\\
&(u-q^2v)B_i(u)B_i(v)+(v-q^2u)B_i(v)B_i(u)\notag\\
&\hskip 2cm =\frac{\delta(uv)(u-v)}{q^{-1}-q}\big(\aTheta_i(u)-\aTheta_i(v)\big),\qquad \text{ if }i=\tau i,\label{rel:BB2}\\
&(u-q^{-1}v)B_i(u)B_{\tau i}(v)+(v-q^{-1}u)B_{\tau i}(v)B_i(u)\notag\\
&\hskip 2cm=\frac{\delta(uv)(u-v)}{q^{2}-1}(\acute{\Theta}_i(u)-\acute{\Theta}_{\tau i}(v)),\qquad\text{if }c_{i,\tau i}=-1,\label{rel:BB3}\\
&[B_i(u),B_j(v)]=0,\qquad \text{if }c_{ij}=0, j\ne \tau i,\label{rel:BB4}\\
&(u-q^{c_{ij}}v)B_i(u)B_j(v)+(v-q^{c_{ij}}u)B_j(v)B_i(u)=0, \quad \text{ if }j\ne \tau i,\label{rel:BB5}
\end{align}
and the Serre relations: for $c_{ij}=-1$,$j\neq \tau i\neq i$,
\beq\label{rel:Serre1}
    \mathbb S_{i,j}(u_1,u_2|v) =0,
\eeq
and  for $c_{ij}=-1$, $i=\tau i$,
\beq\label{rel:Serre2}
\mathbb S_{i,j}(u_1,u_2|v)=
\dfrac{\delta(u_1u_2)(u_1-u_2)v}{(u_1-qv)(u_2-qv)}\big(\aTheta_i(u_1)-\aTheta_i(u_2)\big)B_j(v), 
\eeq
and  for $c_{i,\tau i}=-1$,
\beq\label{rel:Serre3} 
\begin{split}
\mathbb S_{i,\tau i}(u_1,u_2|v)=\Sym_{u_1,u_2}\frac{\delta(u_2v)(1+q^2)(u_2-v)v}{(qu_1-v)(v-q^2u_1^{-1})}\big(\acute{\Theta}_i(u_2)-\acute{\Theta}_{\tau i}(v)\big)B_i(u_1).
\end{split}
\eeq
where
\beq\label{rel:deg}
\aTheta_i(u)=\bK_{\tau i}\Big(u^{\ell_{\tau i}}+\sum_{r>-\ell_{\tau i}}(q-q^{-1})\Theta_{i,r}u^{-r}\Big),\qquad B_i(u)=\sum_{s\in\bZ}B_{i,s}u^{-s}.
\eeq
\end{dfn}
\begin{rem}
The presentation looks different from that for affine iquantum groups as we have done the substitution for the spectral parameter $u\mapsto u^{-1}$. 
\end{rem}

\section{$\mathrm{i}$GKLO Representations of shifted affine $\mathrm{i}$quantum groups}
\label{sec:iGKLO}

In this section, we formulate a family of iGKLO representations of shifted affine iquantum groups of arbitrary quasi-split ADE type, generalizing \cite[Section 7]{FT19} and quantizing \cite[Section~ 3]{LWW25iY}. 

\subsection{Ring of difference operators}
\label{ssec:quantumtorus}

Fix a pair of integral coweights $\lambda\gge \mu$ with $\la$ dominant such that  
\beq
\la-\mu =\sum_{i\in \I}\bv_i\alpha_i^\vee
\eeq
where $\bv_i\in\bN$ for $i\in \I$. Introduce $\theta_i\in\{0,1\}$ for $i\in \I$ subject to the constraints
\beq\label{theta-res}
\theta_i=0, \text{ for } i\in \I\setminus\I_0,\quad \text{ and } \quad c_{ij}\theta_i\theta_j=0,  \text{ for }i\ne j\in\I.
\eeq
Set $\bw_i=\langle\la,\alpha_i\rangle$, for $i\in \I$.

Let $q$ be a formal variable, and we fix a square root $q^{1/2}$ throughout the paper. Let $\bm z:=(z_{i,s})_{i\in\I, 1\lle s\lle \bw_i }$ be formal variables and denote the Laurent polynomial ring
\[
\bC(q^{1/2})[\bm z^{\pm 1}]=\bC(q^{1/2})[z_{i,s}^{\pm 1}]_{i\in\I, 1\lle s\lle \bw_i }.
\]

Consider the associative $\bC[q^{\pm {1/2}}]$-algebra $\widehat{\mathscr A}^q$ generated by $\{\partial_{i,r}^{\pm 1}, w_{i,r}^{\pm1/2}\}_{i\in \I,1\lle r\lle \bv_i}$ with the defining relations
\[
[\partial_{i,r},\partial_{j,s}]=[w_{i,r}^{1/2},w_{j,s}^{1/2}]=0,\quad \partial_{i,r}^{\pm 1}\partial_{i,r}^{\mp 1}=w_{i,r}^{\pm 1/2}w_{i,r}^{\mp 1/2}=1,\quad \partial_{i,r}w_{j,s}^{1/2}=q^{\delta_{ij}\delta_{rs}}w_{j,s}^{1/2}\partial_{i,r},
\]
for all $i,j\in \I$,  $1\lle r\lle \bv_i$, and $1\lle s\lle \bv_j$.
Denote by $\widetilde{\mathscr A}^q$ the localization of $\widehat{\mathscr A}^q$ by the multiplicative set generated by 
$$
\big\{w_{i,r}-q^mw_{i,s}^{\pm 1},w_{i,r}-q^mw_{i,r}^{-1},1-q^mw_{i,r}\big\}_{i\in \I,1\lle r\ne s\lle \bv_i,m\in \bZ}.
$$
Define the $\bC(q^{1/2})$-counterpart
\[
\mathscr A^q:=\widetilde{\mathscr A}^q\otimes_{\bC[q^{\pm {1/2}}]}\bC(q^{1/2}).
\]
Consider the larger algebra
\begin{align*}
\Ui_\mu[\bm z^{\pm 1}]:=\Ui_\mu\otimes_{\bC(q)}\bC(q^{1/2})[\bm z^{\pm 1}],\qquad 
\mathscr A^q[\bm z^{\pm 1}]:=\mathscr A^q\otimes_{\bC(q^{1/2})}\bC(q^{1/2})[\bm z^{\pm 1}],
\end{align*}
with the new central elements $z_{i,s}$.

For $i\in \I$, define
\[
W_i(u):=\prod_{r=1}^{\bv_i}w_{i,r}^{-1/2}\Big(1-\frac{w_{i,r}}{u}\Big), \qquad Z_i(u):=\prod_{r=1}^{\bw_i}\Big(1-\frac{z_{i,r}}{u}\Big),
\]
and set
\[
\W_{i}(u):=W_i(u)W_{\tau i}(u^{-1}), \quad \Z_i(u):=Z_i(u)Z_{\tau i}(u^{-1}).
\]
Clearly, we have
\beq\label{eq:invariant}
\W_{i}(u)=\W_{\tau i}(u^{-1}),\qquad \Z_{i}(u)=\Z_{\tau i}(u^{-1}).
\eeq
For $i\in \I_0$, the leading term of $\W_i(u)$ is central expanding at $u=\infty$. However, for $i\notin \I_0$, the leading term of $\W_i(u)$ is not central in general.
We also need 
\[
\W_{i,r}(u):=W_{\tau i}(u^{-1})\prod_{t=1}^{\bv_i}w_{i,t}^{-1/2}\prod_{s=1,s\ne r}^{\bv_i}\Big(1-\frac{w_{i,s}}{u}\Big),
\]

 \subsection{iGKLO representations}\label{sec:gklo}
Fix an arbitrary orientation of the diagram $\I$ such that for each $i\in \I$ with $i\ne \tau i$, if $i\to j$, then $\tau j\to \tau i$, or if $j\to i$, then $\tau i\to \tau j$. Let
\beq\label{s-def}
\wp_{i}=\begin{cases}
\hf, & \text{if }~i\leftarrow \tau i,\\
-\hf, & \text{if }~i\rightarrow \tau i,\\
0, & \text{if }~c_{i,\tau i}=0, 2.
\end{cases}
\eeq
For simplicity, set
\[
\bm\varkappa(u)=\frac{(1-qu)(1-q^{-1}u)}{(1-u)^2}.
\]
Clearly, $\bm\varkappa(u)=\bm\varkappa(u^{-1})$.

The main result of this paper is the following theorem on the iGKLO representations for shifted affine iquantum groups $\Ui_\mu$ of arbitrary quasi-split ADE types.
This can be viewed as an $\imath$-generalization of \cite[Theorem 3.1]{GKLO05b} and \cite[Theorem 7.1]{FT19} and a $q$-deformation of \cite[Theorem~3.6]{LWW25iY}. The additive factor $\pm w_{i,r}+\tfrac{m}{2}$ in \cite[Theorem~3.6]{LWW25iY} corresponds to the multiplicative factor $w_{i,r}^{\pm 1}q^m$ here. Denote by $|\wp_i|$ the absolute value of $\wp_i$.

\begin{thm} [iGKLO representations]
\label{thmq}
Given scalars $\zeta_i\in\bC(q^{1/2})^\times$ for $i\in \I$, there exists a unique $\bC(q^{1/2})[\bm z^{\pm 1}]$-algebra homomorphism
\[
\Phi_{\mu}^\la: \Ui_\mu [\bm z^{\pm 1}]\longrightarrow \mathscr A^q[\bm z^{\pm 1}],
\]
given by 
\begin{align*}
\aTheta_i(u) &\mapsto \zeta_i\zeta_{\tau i}\brown{\frac{uq^{\wp_i}-(uq^{\wp_i})^{-1}}{u-u^{-1}}}\frac{\bm\varkappa(u)^{\theta_i}\Z_i(u)}{\W_i(qu)\W_{i}(q^{-1}u)}\prod_{j\leftrightarrow i}\W_j(u), \qquad i\in \I,
\\
B_i(u)&\mapsto \frac{\brown{q^{|\wp_i|}}\zeta_i}{1-q^2}
\sum_{r=1}^{\bv_i}\delta\Big(\frac{w_{i,r}}{qu}\Big)\frac{Z_i(u)}{\W_{i,r}(w_{i,r})}\prod_{j\rightarrow i}\W_j(u)\partial_{i,r}^{-1}\\
&\hskip 0.6cm -\frac{q\zeta_i}{1-q^2} \sum_{r=1}^{\bv_{\tau i}}\delta\big(qw_{\tau i,r}u\big)\frac{Z_i(u)}{\W_{\tau i,r}(w_{\tau i,r})}\prod_{\tau j\leftarrow \tau i}\W_{\tau j}(u^{-1})\partial_{\tau  i,r},\qquad i\in \I\setminus\I_0,
\end{align*}
and 
\beq\label{rel:GKLO-split}
\begin{split}
B_i(u)&\mapsto \frac{\zeta_i}{1-q^2}
\sum_{r=1}^{\bv_i}\delta\Big(\frac{w_{i,r}}{qu}\Big)\frac{Z_i(u)}{\W_{i,r}(w_{i,r})}\prod_{j\rightarrow i}\W_j(u)\prod_{j\leftarrow i}^{j=\tau j}W_j(u^{-1})\partial_{i,r}^{-1}\\
&\hskip 0.53cm +\frac{q\zeta_i}{1-q^2} \sum_{r=1}^{\bv_i}\delta\big(qw_{i,r}u\big)\frac{\bm\varkappa(u)^{\theta_i}Z_i(u)}{\W_{i,r}(w_{i,r})}\prod_{ j\leftarrow  i}^{j\ne \tau j}\W_{ j}(u^{-1})\prod_{j\leftarrow i}^{j=\tau j}W_j(u^{-1})\partial_{\tau  i,r}\\
&\hskip 0.53cm +\delta(u)\frac{\sqrt{-1}\,\theta_i\zeta_iZ_i(1)}{(1+q)\W_{i}(q)}\prod_{j\leftrightarrow i}W_j(1),\qquad \qquad i\in \I_0,
\end{split}
\eeq
where the products are over $j\in \I$. 
\end{thm}
Note that the images of $\bK_i$ under $\Phi_{\mu}^\la$ are specified by the images of $\aTheta_i(u)$; see \eqref{rel:deg}.
Theorem~ \ref{thmq} is proved in Section \ref{sec:pf-split} for split type and Section \ref{sec:pf-quasi-split} for non-split type below.

\begin{rem}
We have the following correspondence between extra factors from here to those in \cite[Theorem 3.6]{LWW25iY}:
\[
\frac{(1-qu)(1-q^{-1}u)}{(1-u)^2}\mapsto\frac{(u+\hf)(u-\hf)}{u^2},  \qquad \frac{uq^{\wp_i}-(uq^{\wp_i})^{-1}}{u-u^{-1}}\mapsto 1+\frac{\wp_i}{4u}.
\]
Note that $\wp_i$ from \cite{LWW25iY} is in $\{\pm 1\}$ rather than in $\{\pm\hf\}$ here.
\end{rem}

\section{Proof of Theorem \ref{thmq} for the split type}
\label{sec:pf-split}

In this section, we prove Theorem \ref{thmq} for split type. In this case all $\wp_i=0$ and thus the colored terms in the formulas of the theorem can be dropped. 

\subsection{Relations \eqref{rel:HH} and \eqref{rel:BB4}}

Clearly, if we expand the rational function $\Phi_\mu^\la(\aTheta_i(u))$ at $u=\infty$, it has the desired order of pole at $u=\infty$ as in \eqref{rel:deg}. We need to verify that the relations \eqref{rel:HH}, \eqref{rel:HB}, \eqref{rel:BB2}, \eqref{rel:BB4}, \eqref{rel:BB5} and \eqref{rel:Serre2} are preserved by the map $\Phi_\mu^\la$. The relations \eqref{rel:HH} and \eqref{rel:BB4} are clear. 

We shall verify the remaining relations \eqref{rel:HB}, \eqref{rel:BB2}, \eqref{rel:BB5} and \eqref{rel:Serre2}, respectively.

\subsection{Relation \eqref{rel:HB}}
We need to verify that
\beq\label{q2pf}
\Phi_\mu^\la\big(\aTheta_i(u)B_j(v)\big)=\frac{(q^{c_{ij}}u-v)(q^{c_{ij}}u^{-1}-v)}{(u-q^{c_{ij}}v)(u^{-1}-q^{c_{ij}}v)}\Phi_\mu^\la\big(B_j(v)\aTheta_i(u)\big).
\eeq
Note that $\Phi_\mu^\la\big(\aTheta_i(u)\big)$ acts by scalar multiplication while $\Phi_\mu^\la\big(B_j(v)\big)$ is a linear combination of difference operators $\partial_{j,r}^{\pm 1}$ plus $\delta(v)$ times a scalar rational function. Since
\[
\frac{(q^{c_{ij}}u-v)(q^{c_{ij}}u^{-1}-v)}{(u-q^{c_{ij}}v)(u^{-1}-q^{c_{ij}}v)}\bigg|_{v=1}=1,
\] 
the terms not involving difference operators from both sides of \eqref{q2pf} are clearly equal. Then we consider terms containing difference operators. We shall always move $\partial_{j,r}^{\pm 1}$ to the right and then compare their coefficients. There are 3 cases.
\subsubsection{The case $c_{ij}=0$}
This case is clear.
\subsubsection{The case $c_{ij}=2$} 
We have $i=j$. Let us consider the coefficients of $\partial_{i,r}^{-1}$ in both sides of \eqref{q2pf}. There will be a lot of common factors and hence we only need to focus on the parts that differ for both sides. In this case, the difference comes from the part of $\Phi_\mu^\la(\aTheta_i(u))$ that $\partial_{i,r}^{-1}$ acts non-trivially. Thus it must come from $\W_i(q^{-1}u)\W_i(qu)$. It eventually reduces to showing that
\[
\Big(\Big(1-\frac{w_{i,r}}{qu}\Big) \Big(1-qw_{i,r}u\Big)\Big(1-\frac{qw_{i,r}}{u}\Big) \Big(1-q^{-1}w_{i,r}u\Big)\Big)^{-1}\delta\Big(\frac{w_{i,r}}{qv}\Big)
\]
from $\Phi_\mu^\la\big(\aTheta_i(u)B_j(v)\big)$ is equal to $\dfrac{(q^{2}u-v)(q^2u^{-1}-v)}{(u-q^2v)(u^{-1}-q^2v)}$ times
$$
\brown{q^{-4}}\Big(\Big(1-\frac{w_{i,r}}{q^3u}\Big) \Big(1-q^{-1}w_{i,r}u\Big)\Big(1-\frac{w_{i,r}}{qu}\Big) \Big(1-q^{-3}w_{i,r}u\Big)\Big)^{-1}\delta\Big(\frac{w_{i,r}}{qv}\Big),
$$
from $\Phi_\mu^\la\big(B_j(v)\aTheta_i(u)\big)$. Here $q^{-4}$ comes from moving $\partial_{i,r}^{-1}$ through $w_{i,r}^2$. By the first equality in \eqref{delta}, we can replace $w_{i,r}$ with $qv$. Hence it amounts to proving that
\begin{align*}
\bigg(&\Big(1-\frac{v}{u}\Big) \Big(1-q^2uv\Big)\Big(1-\frac{q^2v}{u}\Big) \Big(1-uv\Big)\bigg)^{-1}\\
&=\brown{q^{-4}}\frac{(q^{2}u-v)(q^2u^{-1}-v)}{(u-q^2v)(u^{-1}-q^2v)}\bigg(\Big(1-\frac{q^{-2}v}{u}\Big) \Big(1-uv\Big)\Big(1-\frac{v}{u}\Big) \Big(1-q^{-2}uv\Big)\bigg)^{-1},
\end{align*}
which clearly holds.

Similarly, to compare the coefficients of $\partial_{i,r}$, we need to prove that
$$
\bigg(\Big(1-\frac{w_{i,r}}{qu}\Big) \Big(1-qw_{i,r}u\Big)\Big(1-\frac{qw_{i,r}}{u}\Big) \Big(1-q^{-1}w_{i,r}u\Big)\bigg)^{-1}\delta\big(qw_{i,r}v\big)
$$ 
from $\Phi_\mu^\la\big(\aTheta_i(u)B_j(v)\big)$ is equal to $\dfrac{(q^{2}u-v)(q^2u^{-1}-v)}{(u-q^2v)(u^{-1}-q^2v)}$ 
times
$$
\brown{q^{4}}\bigg(\Big(1-\frac{qw_{i,r}}{u}\Big) \Big(1-q^3w_{i,r}u\Big)\Big(1-\frac{q^3w_{i,r}}{u}\Big) \Big(1-qw_{i,r}u\Big)\bigg)^{-1}\delta\big(qw_{i,r}v\big),
$$ 
from $\Phi_\mu^\la\big(B_j(v)\aTheta_i(u)\big)$. Here $q^{4}$ comes from moving $\partial_{i,r}$ through $w_{i,r}^2$. By the first equality in \eqref{delta}, we can replace $w_{i,r}$ with $q^{-1}v^{-1}$. Hence it amounts to proving that
\begin{align*}
\bigg(&\Big(1-\frac{1}{q^2uv}\Big) \Big(1-\frac{u}{v}\Big)\Big(1-\frac{1}{uv}\Big) \Big(1-\frac{u}{q^2v}\Big)\bigg)^{-1}\\
&=\brown{q^{4}}\frac{(q^{2}u-v)(q^2u^{-1}-v)}{(u-q^2v)(u^{-1}-q^2v)}\bigg(\Big(1-\frac{1}{uv}\Big) \Big(1-\frac{q^2u}{v}\Big)\Big(1-\frac{q^2}{uv}\Big) \Big(1-\frac{u}{v}\Big)\bigg)^{-1},
\end{align*}
which clearly holds.

\subsubsection{The case $c_{ij}=-1$} 
Let us consider the coefficients of $\partial_{j,r}^{-1}$ in both sides of \eqref{q2pf}. In this case, the difference comes from $\W_j(u)$. Thus it  reduces to showing that
\begin{align*}
&\Big(1-\frac{w_{j,r}}{u}\Big) \Big(1-w_{j,r}u\Big)\delta\Big(\frac{w_{j,r}}{qv}\Big)\\
&=\brown{q^{2}}\frac{(q^{-1}u-v)(q^{-1}u^{-1}-v)}{(u-q^{-1}v)(u^{-1}-q^{-1}v)}\Big(1-\frac{q^{-2}w_{j,r}}{u}\Big) \Big(1-q^{-2}w_{j,r}u\Big)\delta\Big(\frac{w_{j,r}}{qv}\Big),
\end{align*}
where $q^{2}$ comes from moving $\partial_{j,r}^{-1}$ through $w_{j,r}^{-1}$. By the first equality in \eqref{delta}, we can replace $w_{j,r}$ with $qv$. Hence it amounts to verifying that
\begin{align*}
\Big(1-\frac{qv}{u}\Big) \Big(1-quv\Big)=\brown{q^{2}}\frac{(q^{-1}u-v)(q^{-1}u^{-1}-v)}{(u-q^{-1}v)(u^{-1}-q^{-1}v)}\Big(1-\frac{v}{qu}\Big) \Big(1-\frac{uv}{q}\Big),
\end{align*}
which clearly holds. The case for coefficients of $\partial_{j,r}$ is similar and we omit the details.

\subsection{Relation \eqref{rel:BB5}}\label{ssec:BB5}
To simplify the next task, following \cite{GKLO,LWW25iY} we set $w_{i,0}=q$ and
introduce
\begin{align*}
\chi_{i,0}^+=& \frac{\sqrt{-1}\,\theta_i\zeta_iZ_i(1)}{(1+q)\W_{i}(q)}\prod_{j\leftrightarrow i}W_j(1) ,\\
\chi_{i,r}^+=&\frac{\zeta_i}{1-q^2}\frac{Z_i(w_{i,r}/q)}{\W_{i,r}(w_{i,r})}
\prod_{j\rightarrow i}\W_j(w_{i,r}/q)\prod_{j\leftarrow i}W_j(q/w_{i,r})\partial_{i,r}^{-1},\\
\chi_{i,r}^-=&\frac{q\zeta_i}{1-q^2}
\frac{\bm \varkappa(qw_{i,r})^{\theta_i} Z_i(w_{i,r}^{-1}/q)}{\W_{i,r}(w_{i,r})}\prod_{j\leftarrow i}W_j(qw_{i,r})\partial_{i,r},
\end{align*}
for $i\in \I$ and $1\lle r\lle \bv_i$. Note that $\chi_{i,0}^+=0$ when $\theta_i=0$. Then we have
\[
B_i(u)=\sum_{r=0}^{\bv_i} \delta\Big(\frac{w_{i,r}}{qu}\Big)\chi_{i,r}^+ + \sum_{r=1}^{\bv_i}\delta\big( qw_{i,r}u\big)\chi_{i,r}^- .
\]

Now we proceed to prove the relation \eqref{rel:BB5} for the case $c_{ij}=-1$ (as the case for $c_{ij}= 0$ is implied by the easy relation \eqref{rel:BB4}). Note that $\theta_i\theta_j=0$ by \eqref{theta-res}. Without loss of generality, we assume that $\theta_j=0$.
\begin{lem}\label{chirelq}
We have
\begin{align*}
(w_{i,r}^{\pm 1}-q^2w_{i,s}^{\pm 1})\chi_{i,r}^\pm\chi_{i,s}^\pm&=( q^2w_{i,r}^{\pm 1}- w_{i,s}^{\pm 1})\chi_{i,s}^\pm\chi_{i,r}^\pm,\qquad \text{if }r\ne s,\\
(w_{i,r}^{\pm 1}-q^2w_{i,s}^{\mp 1})\chi_{i,r}^\pm\chi_{i,s}^\mp&=( q^2w_{i,r}^{\pm 1}- w_{i,s}^{\mp 1})\chi_{i,s}^\mp\chi_{i,r}^\pm,\qquad \text{if }r\ne s,\\
(w_{i,r}^{\pm 1}-q^{c_{ij}}w_{j,s}^{\pm 1})\chi_{i,r}^\pm\chi_{j,s}^\pm&=( q^{c_{ij}}w_{i,r}^{\pm 1}- w_{j,s}^{\pm 1})\chi_{j,s}^\pm\chi_{i,r}^\pm,\\
(w_{i,r}^{\pm 1}-q^{c_{ij}}w_{j,s}^{\mp 1})\chi_{i,r}^\pm\chi_{j,s}^\mp&=( q^{c_{ij}}w_{i,r}^{\pm 1}- w_{j,s}^{\mp 1})\chi_{j,s}^\mp\chi_{i,r}^\pm,
\end{align*}  
where $r=0$ is allowed in $\chi_{i,r}^+$.
\end{lem}
\begin{proof}
Follows from a direct computation.
\end{proof}
The relation \eqref{rel:BB5} can be equivalently written as
\beq\label{helper03}
(u-q^{c_{ij}}v)B_i(u)B_j(v)=(q^{c_{ij}}u-v)B_j(v)B_i(u).
\eeq
The image of the LHS of \eqref{helper03} under $\Phi_\mu^\la$ is given by
\begin{align*}
    \Phi_{\mu}^\la\big((u-q^{c_{ij}}v&)B_i(u)B_j(v)\big)\\
    = q^{-1}\sum_{r=0}^{\bv_i}\sum_{s=1}^{\bv_j}\bigg(&\delta\Big(\frac{w_{i,r}}{qu}\Big)\delta\Big(\frac{w_{j,s}}{qv}\Big)(w_{i,r}-q^{c_{ij}}w_{j,s})\chi_{i,r}^+\chi_{j,s}^+\\
    +\,&\delta\Big(\frac{w_{i,r}}{qu}\Big)\delta\big(qw_{j,s}v\big)(w_{i,r}-q^{c_{ij}}w_{j,s}^{-1})\chi_{i,r}^+\chi_{j,s}^-\bigg)\\    +q^{-1}\sum_{r=1}^{\bv_i}\sum_{s=1}^{\bv_j}\bigg(&\delta\big(qw_{i,r}u\big)\delta\Big(\frac{w_{j,s}}{qv}\Big)(w_{i,r}^{-1}-q^{c_{ij}}w_{j,s})\chi_{i,r}^-\chi_{j,s}^+\\    +\,&\delta\big(qw_{i,r}u\big)\delta\big(qw_{j,s}v\big)(w_{i,r}^{-1}-q^{c_{ij}}w_{j,s}^{-1})\chi_{i,r}^-\chi_{j,s}^-
    \bigg),
\end{align*}
while the RHS of \eqref{helper03} under $\Phi_\mu^\la$ is given by
\begin{align*}
    \Phi_{\mu}^\la\big((q^{c_{ij}}u-v&)B_j(v)B_i(u)\big)\\
    = q^{-1}\sum_{r=0}^{\bv_i}\sum_{s=1}^{\bv_j}\bigg(&\delta\Big(\frac{w_{i,r}}{qu}\Big)\delta\Big(\frac{w_{j,s}}{qv}\Big)(q^{c_{ij}}w_{i,r}-w_{j,s})\chi_{j,s}^+\chi_{i,r}^+\\
    +\,&\delta\Big(\frac{w_{i,r}}{qu}\Big)\delta\big(qw_{j,s}v\big)(q^{c_{ij}}w_{i,r}-w_{j,s}^{-1})\chi_{j,s}^-\chi_{i,r}^+\bigg)\\    +q^{-1}\sum_{r=1}^{\bv_i}\sum_{s=1}^{\bv_j}\bigg(&\delta\big(qw_{i,r}u\big)\delta\Big(\frac{w_{j,s}}{qv}\Big)(q^{c_{ij}}w_{i,r}^{-1}-w_{j,s})\chi_{j,s}^+\chi_{i,r}^-\\    +\,&\delta\big(qw_{i,r}u\big)\delta\big(qw_{j,s}v\big)(q^{c_{ij}}w_{i,r}^{-1}-w_{j,s}^{-1})\chi_{j,s}^-\chi_{i,r}^-
    \bigg).
\end{align*}
It follows from Lemma \ref{chirelq} that
\[
\Phi_{\mu}^\la\big((u-q^{c_{ij}}v)B_i(u)B_j(v)\big)=\Phi_{\mu}^\la\big((q^{c_{ij}}u-v)B_j(v)B_i(u)\big),
\]
whence the relation \eqref{rel:BB5} (also including the relation \eqref{rel:BB4} as a special case).

\subsection{Relation \eqref{rel:BB2}} \label{sec:relBB2}

We start with computing $\Phi_\mu^\la\big((u-q^2v)B_i(u)B_i(v)\big)$, which is given by
\begin{align}
&\sum_{r=1}^{\bv_i}(u-q^2v)\delta\Big(\frac{w_{i,r}}{qu}\Big)\delta\Big(\frac{w_{i,r}v}{q}\Big)\chi_{i,r}^+\chi_{i,r}^- + \sum_{r=1}^{\bv_i}(u-q^2v)\delta\big(qw_{i,r}u\big)\delta\Big(\frac{qw_{i,r}}{v}\Big)\chi_{i,r}^-\chi_{i,r}^+ \nonumber\\
&+\sum_{r=1}^{\bv_i}(u-q^2v)\delta\Big(\frac{w_{i,r}}{qu}\Big)\delta\Big(\frac{w_{i,r}}{q^3v}\Big)\chi_{i,r}^+\chi_{i,r}^+   
+ \sum_{0\lle r\ne s\lle \bv_i}(u-q^2v)\delta\Big(\frac{w_{i,r}}{qu}\Big)\delta\Big(\frac{w_{i,s}}{qv}\Big)\chi_{i,r}^+\chi_{i,s}^+ \label{q5pf1}\\
&+\sum_{\substack{r=0,s=1\\ r \ne s}}^{\bv_i}(u-q^2v)\delta\Big(\frac{w_{i,r}}{qu}\Big)\delta\big(qw_{i,s}v\big)\chi_{i,r}^+\chi_{i,s}^-   
+ \sum_{\substack{r=1,s=0\\ r \ne s}}^{\bv_i}(u-q^2v)\delta\big(qw_{i,r}u\big)\delta\Big(\frac{w_{i,s}}{qv}\Big)\chi_{i,r}^-\chi_{i,s}^+ \label{q5pf0}\\
&+\sum_{r=1}^{\bv_i}(u-q^2v)\delta\big(qw_{i,r}u\big)\delta\big(q^3w_{i,r}v\big)\chi_{i,r}^-\chi_{i,r}^-   
\label{q5pf2}\\
&+ \sum_{1\lle r\ne s\lle \bv_i}(u-q^2v)\delta\big(qw_{i,r}u\big)\delta\big(qw_{i,s}v\big)\chi_{i,r}^-\chi_{i,s}^-+(u-q^2v) \delta(u)\delta(v)\chi_{i,0}^+\chi_{i,0}^+.\label{q5pf22}
\end{align}
By the identities in \eqref{delta}, the first sums in \eqref{q5pf1} and \eqref{q5pf2} are zero. We also obtain a similar equality for $\Phi_\mu^\la\big((v-q^2u)B_i(v)B_i(u)\big)$.

It follows from \eqref{delta} and Lemma \ref{chirelq} that
\begin{align*}
&\sum_{0\lle r\ne s\lle \bv_i}(u-q^2v)\delta\Big(\frac{w_{i,r}}{qu}\Big)\delta\Big(\frac{w_{i,s}}{qv}\Big)\chi_{i,r}^+\chi_{i,s}^+ + \sum_{0\lle r\ne s\lle \bv_i}(v-q^2u)\delta\Big(\frac{w_{i,s}}{qv}\Big)\delta\Big(\frac{w_{i,r}}{qu}\Big)\chi_{i,s}^+\chi_{i,r}^+\\
=&q^{-1}\sum_{0\lle r\ne s\lle \bv_i}\bigg((w_{i,r}-q^2w_{i,s})\delta\Big(\frac{w_{i,r}}{qu}\Big)\delta\Big(\frac{w_{i,s}}{qv}\Big)\chi_{i,r}^+\chi_{i,s}^+\\
&\hskip5.7cm + (w_{i,s}-q^2w_{i,r})\delta\Big(\frac{w_{i,r}}{qu}\Big)\delta\Big(\frac{w_{i,s}}{qv}\Big)\chi_{i,s}^+\chi_{i,r}^+\bigg)\\
=&q^{-1}\sum_{0\lle r\ne s\lle \bv_i}\delta\Big(\frac{w_{i,r}}{qu}\Big)\delta\Big(\frac{w_{i,s}}{qv}\Big)\Big((w_{i,r}-q^2w_{i,s})\chi_{i,r}^+\chi_{i,s}^+
+ (w_{i,s}-q^2w_{i,r})\chi_{i,s}^+\chi_{i,r}^+\Big)=0.
\end{align*}
Similar cancellations also occur for other sums from \eqref{q5pf1}--\eqref{q5pf22}. Thus we find that
\beq\label{topf0}
\begin{split}
&\Phi_\mu^\la\big((u-q^2v)B_i(u)B_i(v)+(v-q^2u)B_i(v)B_i(u)\big)\\
=&\,(u+v)(1-q^2)\delta(u)\delta(v)\chi_{i,0}^+\chi_{i,0}^+\\
&+\sum_{r=1}^{\bv_i}(u-q^2v)\delta\Big(\frac{w_{i,r}}{qu}\Big)\delta\Big(\frac{w_{i,r}v}{q}\Big)\chi_{i,r}^+\chi_{i,r}^- + \sum_{r=1}^{\bv_i}(u-q^2v)\delta\big(qw_{i,r}u\big)\delta\Big(\frac{qw_{i,r}}{v}\Big)\chi_{i,r}^-\chi_{i,r}^+\\
&+\sum_{r=1}^{\bv_i}(v-q^2u)\delta\Big(\frac{w_{i,r}}{qv}\Big)\delta\Big(\frac{w_{i,r}u}{q}\Big)\chi_{i,r}^+\chi_{i,r}^- + \sum_{r=1}^{\bv_i}(v-q^2u)\delta\big(qw_{i,r}v\big)\delta\Big(\frac{qw_{i,r}}{u}\Big)\chi_{i,r}^-\chi_{i,r}^+.
\end{split}
\eeq

To evaluate the RHS, we need the following standard statement, see e.g. \cite[Lemma C.1]{FT19}. For any rational function $\gamma(u)$, denote by $\gamma(u)^+$ and $\gamma(u)^-$ the expansions of $\gamma(u)$ at $u=\infty$ and $u=0$, respectively. 

\begin{lem}\label{exp}
For any rational function $\gamma(u)$ with simple poles $\{a_k\}\subset \bC^\times$ and possibly poles of higher order at $u=0,\infty$, we have
\[
\gamma(u)^+-\gamma(u)^-=\sum_{k}\delta\Big(\frac{a_k}{u}\Big)\mathrm{Res}_{u=a_k}\frac{1}{u}\gamma(u).
\]
\end{lem}

The rational function 
$$
\Xi_i(u):=\Phi_\mu^\la(\aTheta_i(u))=\zeta_i^2 \frac{\bm\varkappa(u)^{\theta_i}\Z_i(u)}{\W_i(qu)\W_{i}(q^{-1}u)}\prod_{j\leftrightarrow i}\W_j(u),$$ 
clearly has the property $\Xi_i(u)=\Xi_i(u^{-1})$.

Note that $\aTheta_i(u)$ is a formal series in $u^{-1}$ and $\delta(uv)\aTheta_i(v)=\delta(uv)\aTheta_i(u^{-1})$. Thus we have
$$\Phi_\mu^\la\big(\delta(uv)(u-v)(\aTheta_i(u)-\aTheta_i(v))\big)=\delta(uv)(u-u^{-1})\big(\Xi_i(u)^+-\Xi_i(u)^-\big).
$$
Since $(u-u^{-1})\Xi_i(u)$ has at most a simple pole at $u=1$, by Lemma \ref{exp}, the image of the RHS of \eqref{rel:BB2} under $\Phi_\mu^\la$ is given by
\beq\label{topf1}
\begin{split}
 \frac{\zeta_i^2 }{q^{-1}-q}&\delta(uv)\frac{-2\theta_i(1-q)^2\Z_i(1)\delta(u)}{q\W_i(q)^2}\prod_{j\leftrightarrow i}\W_j(1)+\sum_{r=1}^{\bv_i}\frac{\zeta_i^2(u-v)}{q^{-1}-q}\delta(uv)\times \\
 &\bigg(\frac{1}{1-q^2} \frac{\bm \varkappa(q^{-1}w_{i,r})^{\theta_i}\Z_i(q^{-1}w_{i,r})}{\W_{i,r}(w_{i,r})\W_{i,r}(q^{-2}w_{i,r})}\prod_{j\leftrightarrow i}\W_j(q^{-1}w_{i,r})\Big(\delta\Big(\frac{w_{i,r}}{qu}\Big)-\delta\Big(\frac{w_{i,r}u}{q}\Big)\Big)\\
 &+\frac{1}{1-q^{-2}} \frac{\bm \varkappa(qw_{i,r})^{\theta_i}\Z_i(qw_{i,r})}{\W_{i,r}(w_{i,r})\W_{i,r}(q^{2}w_{i,r})}\prod_{j\leftrightarrow i}\W_j(qw_{i,r})\Big(\delta\Big(\frac{qw_{i,r}}{u}\Big)-\delta\big(qw_{i,r}u\big)\Big) \bigg).
\end{split}
\eeq

We claim that the 1st, 2nd, 3rd, 4th, 5th sums in \eqref{topf0} are equal to the 1st, 2nd, 5th, 3rd, 4th sums in \eqref{topf1}, respectively. The equality between the first terms is clear. For the remaining case, we only prove the claim for two cases. We first evaluate the following
\begin{align*}
 &(u-q^2v)\delta\Big(\frac{w_{i,r}}{qu}\Big)\delta\Big(\frac{w_{i,r}v}{q}\Big)\chi_{i,r}^+\chi_{i,r}^- 
 \\
 =\,& \frac{\zeta_i^2(u-q^2v)}{q(1-q^2)^2}\delta\Big(\frac{w_{i,r}}{qu}\Big)\delta\Big(\frac{w_{i,r}v}{q}\Big)\frac{\bm \varkappa(q^{-1}w_{i,r})^{\theta_i}\Z_i(q^{-1}w_{i,r})}{\W_{i,r}(w_{i,r})\W_{i,r}(q^{-2}w_{i,r})}\brown{\frac{1-q^{-2}w_{i,r}^2}{1-q^{-4}w_{i,r}^2}}\prod_{j\leftrightarrow i}\W_j(q^{-1}w_{i,r}).
\end{align*}
Note that we have an extra factor $q^{-2}$ as we move $\pa^{-1}_{i,r}$ through $w_{i,r}$. Thus to prove it is equal to the 1st term in \eqref{topf1}, it suffices to verify that
\begin{align*}
\frac{q^{-1}w_{i,r}-q^3w_{i,r}^{-1}}{q(1-q^2)^2}\frac{1-q^{-2}w_{i,r}^2}{1-q^{-4}w_{i,r}^2}=\frac{q^{-1}w_{i,r}-qw_{i,r}^{-1}}{q^{-1}-q}\frac{1}{1-q^2},
\end{align*}
which clearly holds. Then we consider another situation,
\begin{align*}
&(v-q^2u)\delta\big(qw_{i,r}v\big)\delta\Big(\frac{qw_{i,r}}{u}\Big)\chi_{i,r}^-\chi_{i,r}^+ \\
=&\frac{\zeta_i^2(v-q^2u)\brown{q^3}}{(1-q^2)^2}\delta\big(qw_{i,r}v\big)\delta\Big(\frac{qw_{i,r}}{u}\Big)\frac{\Z_i(qw_{i,r})}{\W_{i,r}(w_{i,r})\W_{i,r}(q^{2}w_{i,r})}\brown{\frac{1-q^{2}w_{i,r}^2}{1-q^{4}w_{i,r}^2}}\prod_{j\leftrightarrow i}\W_j(qw_{i,r}).
\end{align*}
Note that we have an extra factor $q^2$ as we move $\pa_{i,r}$ through $w_{i,r}$. Hence to prove it is identical to the 3rd term in \eqref{topf1}, we need to check that
\[
\frac{(q^{-1}w_{i,r}^{-1}-q^3w_{i,r})q^3}{(1-q^2)^2}\frac{1-q^{2}w_{i,r}^2}{1-q^{4}w_{i,r}^2}=\frac{qw_{i,r}-q^{-1}w_{i,r}^{-1}}{q^{-1}-q}\frac{1}{1-q^{-2}},
\]
which is straightforward. By similar computations, we prove that the 2nd, 3rd, 4th, 5th sums in \eqref{topf0} are equal to the 2nd, 5th, 3rd, 4th sums in \eqref{topf1}, proving the relation \eqref{rel:BB2}.

\subsection{The Serre relation \eqref{rel:Serre2}} 
\label{sec:serre-split}

It is convenient to set $\chi_{k,0}^-=0$ and $w_{k,r}^\pm:=w_{k,r}^{\pm 1}$ for $k\in \I$ and $1\lle r\lle \bv_k$. Let us consider the terms appearing in 
\beq\label{ser1}
\begin{split}
&\mathrm{Sym}_{u_1,u_2}\big[\Phi_{\mu}^\la(B_i(u_1)),[\Phi_{\mu}^\la(B_i(u_2)),\Phi_{\mu}^\la(B_j(v))]_q\big]_{q^{-1}}
\\
=\,&\mathrm{Sym}_{u_1,u_2}\sum_{\ka_1,\ka_2,\ka\in\{\pm\}}\sum_{r_1,r_2=0}^{\bv_i}\sum_{s=0}^{\bv_j}\bigg[\delta\Big(\frac{w_{i,r_1}^{\ka_1}}{qu_1}\Big)\chi_{i,r_1}^{\ka_1},\Big[\delta\Big(\frac{w_{i,r_2}^{\ka_2}}{qu_2}\Big)\chi_{i,r_2}^{\ka_2},\delta\Big(\frac{w_{j,s}^{\ka}}{qv}\Big)\chi_{j,s}^{\ka}\Big]_q\bigg]_{q^{-1}}.
\end{split}
\eeq
Due to Lemma \ref{chirelq}, it follows from  similar computations as in \cite[C(vii)]{FT19} that
\begin{align*}
\bigg[\delta\Big(\frac{w_{i,r_1}^{\ka_1}}{qu_1}\Big)\chi_{i,r_1}^{\ka_1},&\Big[\delta\Big(\frac{w_{i,r_2}^{\ka_2}}{qu_2}\Big)\chi_{i,r_2}^{\ka_2},\delta\Big(\frac{w_{j,s}^{\ka}}{qv}\Big)\chi_{j,s}^{\ka}\Big]_q\bigg]_{q^{-1}}\\
+&\bigg[\delta\Big(\frac{w_{i,r_2}^{\ka_2}}{qu_2}\Big)\chi_{i,r_2}^{\ka_2},\Big[\delta\Big(\frac{w_{i,r_1}^{\ka_1}}{qu_1}\Big)\chi_{i,r_1}^ {\ka_1},\delta\Big(\frac{w_{j,s}^{\ka}}{qv}\Big)\chi_{j,s}^{\ka}\Big]_q\bigg]_{q^{-1}}=0   
\end{align*}
for the cases ($r_1\ne r_2$) and ($r_1=r_2\ne 0$ and $\ka_1=\ka_2$), see also below for a similar calculation. Hence \eqref{ser1} is equal to
\beq\label{ser2}
\begin{split}
&\mathrm{Sym}_{u_1,u_2}\sum_{\ka'\in\{\pm\}}\sum_{s=0}^{\bv_j}\bigg[\delta(u_1)\chi_{i,0}^+,\Big[\delta(u_2)\chi_{i,0}^+,\delta\Big(\frac{w_{j,s}^{\ka'}}{qv}\Big)\chi_{j,s}^{\ka'}\Big]_q\bigg]_{q^{-1}}\\
+\,&\mathrm{Sym}_{u_1,u_2}\sum_{\ka,\ka'\in\{\pm\}}\sum_{r=1}^{\bv_i}\sum_{s=0}^{\bv_j}\bigg[\delta\Big(\frac{w_{i,r}^{\ka}}{qu_1}\Big)\chi_{i,r}^{\ka},\Big[\delta\Big(\frac{w_{i,r}^{-\ka}}{qu_2}\Big)\chi_{i,r}^{-\ka},\delta\Big(\frac{w_{j,s}^{\ka'}}{qv}\Big)\chi_{j,s}^{\ka'}\Big]_q\bigg]_{q^{-1}}.	
\end{split}
\eeq
For $1\lle r\lle \bv_i$, by Lemma \ref{chirelq}, we have
\begin{align*}
&\bigg[\delta\Big(\frac{w_{i,r}^{\ka}}{qu_1}\Big)\chi_{i,r}^{\ka},\Big[\delta\Big(\frac{w_{i,r}^{-\ka}}{qu_2}\Big)\chi_{i,r}^{-\ka},\delta\Big(\frac{w_{j,s}^{\ka'}}{qv}\Big)\chi_{j,s}^{\ka'}\Big]_q\bigg]_{q^{-1}}    \\
=& \Big[\delta\Big(\frac{w_{i,r}^{\ka}}{qu_1}\Big)\chi_{i,r}^{\ka},\delta\Big(\frac{w_{i,r}^{-\ka}}{qu_2}\Big)\delta\Big(\frac{w_{j,s}^{\ka'}}{qv}\Big)\Big(1-q\frac{w_{i,r}^{-\ka}-q^{-1}w_{j,s}^{\ka'}}{q^{-1}w_{i,r}^{-\ka}-w_{j,s}^{\ka'}}\Big)\chi_{i,r}^{-\ka}\chi_{j,s}^{\ka'}\Big]_{q^{-1}}\\
=&\Big[\delta\Big(\frac{w_{i,r}^{\ka}}{qu_1}\Big)\chi_{i,r}^{\ka},\delta\Big(\frac{w_{i,r}^{-\ka}}{qu_2}\Big)\delta\Big(\frac{w_{j,s}^{\ka'}}{qv}\Big)\frac{(q^{-1}-q)w_{i,r}^{-\ka}}{q^{-1}w_{i,r}^{-\ka}-w_{j,s}^{\ka'}}\chi_{i,r}^{-\ka}\chi_{j,s}^{\ka'}\Big]_{q^{-1}}\\
=&\delta\Big(\frac{w_{i,r}^{\ka}}{qu_1}\Big)\delta\Big(\frac{qw_{i,r}^{-\ka}}{u_2}\Big)\delta\Big(\frac{w_{j,s}^{\ka'}}{qv}\Big)\frac{(q^{-1}-q)q^2w_{i,r}^{-\ka}}{qw_{i,r}^{-\ka}-w_{j,s}^{\ka'}}\chi_{i,r}^{\ka}\chi_{i,r}^{-\ka}\chi_{j,s}^{\ka'}\\
&\hskip3cm -q^{-1}\delta\Big(\frac{qw_{i,r}^{\ka}}{u_1}\Big)\delta\Big(\frac{w_{i,r}^{-\ka}}{qu_2}\Big)\delta\Big(\frac{w_{j,s}^{\ka'}}{qv}\Big)\frac{(q^{-1}-q)w_{i,r}^{-\ka}}{q^{-1}w_{i,r}^{-\ka}-w_{j,s}^{\ka'}}\chi_{i,r}^{-\ka}\chi_{j,s}^{\ka'}\chi_{i,r}^{\ka}\\
=& \delta\Big(\frac{w_{i,r}^{\ka}}{qu_1}\Big)\delta\Big(\frac{qw_{i,r}^{-\ka}}{u_2}\Big)\delta\Big(\frac{w_{j,s}^{\ka'}}{qv}\Big)\frac{(q^{-1}-q)q^2w_{i,r}^{-\ka}}{qw_{i,r}^{-\ka}-w_{j,s}^{\ka'}}\brown{\chi_{i,r}^{\ka}\chi_{i,r}^{-\ka}\chi_{j,s}^{\ka'}}\\
&\qquad -q^{-1}\delta\Big(\frac{qw_{i,r}^{\ka}}{u_1}\Big)\delta\Big(\frac{w_{i,r}^{-\ka}}{qu_2}\Big)\delta\Big(\frac{w_{j,s}^{\ka'}}{qv}\Big)\frac{(q^{-1}-q)w_{i,r}^{-\ka}}{q^{-1}w_{i,r}^{-\ka}-w_{j,s}^{\ka'}}\frac{q^{2}w_{i,r}^{\ka}-q^{-1}w_{j,s}^{\ka'}}{qw_{i,r}^{\ka}-w_{j,s}^{\ka'}}\chi_{i,r}^{-\ka}\chi_{i,r}^{\ka}\chi_{j,s}^{\ka'}.
\end{align*}
Similarly, we have
\begin{align*}
&\bigg[\delta\Big(\frac{w_{i,r}^{-\ka}}{qu_2}\Big)\chi_{i,r}^{-\ka},\Big[\delta\Big(\frac{w_{i,r}^{\ka}}{qu_1}\Big)\chi_{i,r}^{\ka},\delta\Big(\frac{w_{j,s}^{\ka'}}{qv}\Big)\chi_{j,s}^{\ka'}\Big]_q\bigg]_{q^{-1}} \\
=& \delta\Big(\frac{qw_{i,r}^{\ka}}{u_1}\Big)\delta\Big(\frac{w_{i,r}^{-\ka}}{qu_2}\Big)\delta\Big(\frac{w_{j,s}^{\ka'}}{qv}\Big)\frac{(q^{-1}-q)q^2w_{i,r}^{\ka}}{qw_{i,r}^{\ka}-w_{j,s}^{\ka'}}\chi_{i,r}^{-\ka}\chi_{i,r}^{\ka}\chi_{j,s}^{\ka'}\\
&\qquad -q^{-1}\delta\Big(\frac{w_{i,r}^{\ka}}{qu_1}\Big)\delta\Big(\frac{qw_{i,r}^{-\ka}}{u_2}\Big)\delta\Big(\frac{w_{j,s}^{\ka'}}{qv}\Big)\frac{(q^{-1}-q)w_{i,r}^{\ka}}{q^{-1}w_{i,r}^{\ka}-w_{j,s}^{\ka'}}\frac{q^{2}w_{i,r}^{-\ka}-q^{-1}w_{j,s}^{\ka'}}{qw_{i,r}^{-\ka}-w_{j,s}^{\ka'}}\brown{\chi_{i,r}^{\ka}\chi_{i,r}^{-\ka}\chi_{j,s}^{\ka'}}.
\end{align*}
Note that for fixed $i,j,r,s,\ka,\ka'$, the term
\[
\delta\Big(\frac{w_{i,r}^{\ka}}{qu_1}\Big)\delta\Big(\frac{qw_{i,r}^{-\ka}}{u_2}\Big)\delta\Big(\frac{w_{j,s}^{\ka'}}{qv}\Big) \chi_{i,r}^{\ka}\chi_{i,r}^{-\ka}\chi_{j,s}^{\ka'}
\]
only appears in the above two commutators. We add up the coefficients and simplify:
\begin{align*}
 \frac{(q^{-1}-q)q^2w_{i,r}^{-\ka}}{qw_{i,r}^{-\ka}-w_{j,s}^{\ka'}}-   q^{-1}\frac{(q^{-1}-q)w_{i,r}^{\ka}}{q^{-1}w_{i,r}^{\ka}-w_{j,s}^{\ka'}}\frac{q^{2}w_{i,r}^{-\ka}-q^{-1}w_{j,s}^{\ka'}}{qw_{i,r}^{-\ka}-w_{j,s}^{\ka'}}=\frac{(q^{-1}-q)(q^{-2}w_{i,r}^{\ka}-q^2w_{i,r}^{-\ka})w_{j,s}^{\ka'}}{(q^{-1}w_{i,r}^{\ka}-w_{j,s}^{\ka'})(qw_{i,r}^{-\ka}-w_{j,s}^{\ka'})}.
\end{align*}
On the other hand, we also have
\beq
\begin{split}
&\Big[\delta(u_1)\chi_{i,0}^+,\Big[\delta(u_2)\chi_{i,0}^+,\delta\Big(\frac{w_{j,s}^{\ka'}}{qv}\Big)\chi_{j,s}^{\ka'}\Big]_q\Big]_{q^{-1}}\\
&=\frac{(q^{-1}-q)(1-q^2)q^{-1}w_{j,s}^{\ka'}}{(1-w_{j,s}^{\ka'})^2}\delta(u_1)\delta(u_2)\delta\Big(\frac{w_{j,s}^{\ka'}}{qv}\Big)\big(\chi_{i,0}^+\big)^2\chi_{j,s}^{\ka'}.
\end{split}
\eeq

Hence, \eqref{ser2} can be transformed as
\begin{align*}
&\sum_{\ka\in\{\pm\}}\sum_{s=0}^{\bv_j}\frac{2(q^{-1}-q)(1-q^2)q^{-1}w_{j,s}^{\ka}}{(1-w_{j,s}^{\ka})^2}\delta(u_1)\delta(u_2)\delta\Big(\frac{w_{j,s}^{\ka}}{qv}\Big)\big(\chi_{i,0}^+\big)^2\chi_{j,s}^{\ka}\\
+&\sum_{\ka\in\{\pm\}}\sum_{r=1}^{\bv_i}\sum_{s=0}^{\bv_j}\bigg(
\frac{(q^{-1}-q)(q^{-2}w_{i,r}-q^2w_{i,r}^{-1})w_{j,s}^{\ka}}{(q^{-1}w_{i,r}-w_{j,s}^{\ka})(qw_{i,r}^{-1}-w_{j,s}^{\ka})}
\delta\Big(\frac{w_{i,r}}{qu_1}\Big)\delta\Big(\frac{qw_{i,r}^{-1}}{u_2}\Big)\delta\Big(\frac{w_{j,s}^{\ka}}{qv}\Big) \chi_{i,r}^+\chi_{i,r}^{-}\chi_{j,s}^{\ka}\\
&\hskip2.2cm+\frac{(q^{-1}-q)(q^{-2}w_{i,r}-q^2w_{i,r}^{-1})w_{j,s}^{\ka}}{(q^{-1}w_{i,r}-w_{j,s}^{\ka})(qw_{i,r}^{-1}-w_{j,s}^{\ka})}
\delta\Big(\frac{w_{i,r}}{qu_2}\Big)\delta\Big(\frac{qw_{i,r}^{-1}}{u_1}\Big)\delta\Big(\frac{w_{j,s}^{\ka}}{qv}\Big) \chi_{i,r}^+\chi_{i,r}^{-}\chi_{j,s}^{\ka}\\
&\hskip2.2cm+\frac{(q^{-1}-q)(q^{-2}w_{i,r}^{-1}-q^2w_{i,r})w_{j,s}^{\ka}}{(qw_{i,r}-w_{j,s}^{\ka})(q^{-1}w_{i,r}^{-1}-w_{j,s}^{\ka})}
\delta\Big(\frac{qw_{i,r}}{u_1}\Big)\delta\Big(\frac{w_{i,r}^{-1}}{qu_2}\Big)\delta\Big(\frac{w_{j,s}^{\ka}}{qv}\Big) \chi_{i,r}^-\chi_{i,r}^{+}\chi_{j,s}^{\ka}\\
&\hskip2.2cm+\frac{(q^{-1}-q)(q^{-2}w_{i,r}^{-1}-q^2w_{i,r})w_{j,s}^{\ka}}{(qw_{i,r}-w_{j,s}^{\ka})(q^{-1}w_{i,r}^{-1}-w_{j,s}^{\ka})}
\delta\Big(\frac{qw_{i,r}}{u_2}\Big)\delta\Big(\frac{w_{i,r}^{-1}}{qu_1}\Big)\delta\Big(\frac{w_{j,s}^{\ka}}{qv}\Big) \chi_{i,r}^-\chi_{i,r}^{+}\chi_{j,s}^{\ka}
\bigg),
\end{align*}
cf. the RHS of \eqref{topf0}. In order to show that the above is equal to the image of the RHS of \eqref{rel:Serre2} under $\Phi_\mu^\la$, one only needs to proceed as in \S\ref{sec:relBB2} by using Lemma \ref{exp}. Note that we also applied identities such as
\[
\frac{q^{-1}w_{j,s}^{\ka}}{(1-w_{j,s}^{\ka})^2}\delta(u_1)\delta(u_2)\delta\Big(\frac{w_{j,s}^{\ka}}{qv}\Big)=\frac{v\delta(u_1)\delta(u_2)}{(u_1-qv)(u_2-qv)}\delta\Big(\frac{w_{j,s}^{\ka}}{qv}\Big)
\]
and
\begin{align*}
\frac{(q^{-1}-q)(q^{-2}w_{i,r}-q^2w_{i,r}^{-1})w_{j,s}^{\ka}}{(q^{-1}w_{i,r}-w_{j,s}^{\ka})(qw_{i,r}^{-1}-w_{j,s}^{\ka})}&
\delta\Big(\frac{w_{i,r}}{qu_1}\Big)\delta\Big(\frac{qw_{i,r}^{-1}}{u_2}\Big)\delta\Big(\frac{w_{j,s}^{\ka}}{qv}\Big)\\=&\,\frac{(q^{-1}-q)(u_1-q^2u_2)v}{(u_1-qv)(u_2-qv)}
\delta\Big(\frac{w_{i,r}}{qu_1}\Big)\delta\Big(\frac{qw_{i,r}^{-1}}{u_2}\Big)\delta\Big(\frac{w_{j,s}^{\ka}}{qv}\Big).  \end{align*}

\section{Proof of Theorem \ref{thmq} for the non-split type}
\label{sec:pf-quasi-split}

In this section, we prove Theorem \ref{thmq} for the non-split type (i.e., with $\tau\ne \mathrm{id}$).

\subsection{Proof for quasi-split type excluding type A$_{2n}$}
\label{ssec:qspf}

In this subsection, we exclude the quasi-split type A$_{2n}$. Hence $c_{i,\tau i}\ne -1$ for all $i\in \I$, and the Serre relation \eqref{rel:Serre3} does not occur and $\wp_i=0$ for all $i\in\I$. It is not hard to see that the verification of all the relations will be either similar to the case for the ordinary shifted affine quantum groups for $c_{i,\tau i}=0$, or the case for shifted affine iquantum groups of split type for $i=\tau i$. Here we only discuss for example the Serre relation for $c_{ij}=-1$ as the relation for $c_{ij}=0$ reduces to the case of shifted affine quantum groups.

Suppose $c_{ij}=-1$. Note that $j\ne \tau i$, then we have several cases. First, one needs to verify that similar identities hold as in Lemma \ref{chirelq}. We proceed case-by-case. 
\begin{enumerate}
    \item If both $i$ and $j$ are not fixed by $\tau$, then this is the case of ordinary shifted affine quantum groups as verified in \cite[Appendix C]{FT19}.
    \item If both $i$ and $j$ are fixed by $\tau$, then this is the split case as verified in \S\ref{sec:serre-split}.
    \item If $i$ is fixed by $\tau$ and $j$ is not, i.e., $i\in \I_0$ and $j\notin \I_0$, then we need to verify the relation \eqref{rel:Serre2}. The details are parallel to that of \S\ref{sec:serre-split}. Again one needs to carefully deal with the terms containing $\chi_{i,0}^+\chi_{i,0}^+$ and $\chi_{i,r}^\pm\chi_{i,r}^\mp$ (those are scalar functions) while all other terms cancel due to Lemma \ref{chirelq}.
    \item If $j$ is fixed by $\tau$ and $i$ is not, i.e., $i\notin \I_0$ and $j\in \I_0$, then we need to verify the relation \eqref{rel:Serre1}. Again, we argue as in \S\ref{sec:serre-split}. Since $i$ is not fixed by $\tau$ and $j\ne \tau i$, there is no way to obtain constant terms (scalar rational functions without difference operators). Thus this case essentially corresponds to the beginning of \S\ref{sec:serre-split}.
\end{enumerate}

\subsection{Completing the proof for the quasi-split type A$_{2n}$}
\label{ssec:qsAIII2n}

In this subsection, we complete the proof of Theorem \ref{thmq} for the case of quasi-split type A$_{2n}$. We shall only verify the most complicated relations \eqref{rel:BB3} and \eqref{rel:Serre3} for the case $j=\tau i$ and $c_{i,\tau i}=-1$ for the iGKLO homomorphism in Theorem \ref{thmq}. In this case, we have $\theta_i=\theta_j=0$. We prove the case $i\rightarrow j$ (as the other case $i\leftarrow j$ is similar). Then $\wp_i=-\hf$ and $\wp_j=\hf$. 

Set $n_i=\bv_i+\bv_j$. For $1\lle r\lle n_i$, denote 
$$
r':= n_i+1-r.
$$ 
As in \cite[\S3.1]{LWW25iY}, it is convenient to extend the notation $w_{k,r}$ and $\partial_{k,r}^{\pm 1}$  to $1\lle r \lle n_i$ for $k\in\{i,j\}$ by the rule:
\beq\label{eq:ext}
w_{k,r}:=w_{\tau k,r'}^{-1},\qquad \qquad \partial_{k,r}:=\partial_{\tau k,r'}^{-1}.
\eeq

\subsubsection{The relation \eqref{rel:BB3}} 
\label{ssec:BB3}
We proceed as in \S\ref{ssec:BB5} by introducing $\chi_{i,r}$ and $\chi_{j,s}$, for $1\lle r,s\lle n_i$, such that
\[
\Phi_\mu^\la(B_i(u))=\sum_{r=1}^{n_i}\delta\Big(\frac{w_{i,r}}{qu}\Big)\chi_{i,r},\qquad \Phi_\mu^\la(B_j(v))=\sum_{s=1}^{n_i}\delta\Big(\frac{w_{j,s}}{qv}\Big)\chi_{j,s}.
\]

\begin{lem}
 \label{lem:qs}
For $r\ne s$, we have
\begin{align*}
(w_{i,r}-q^2w_{i,s})\chi_{i,r}\chi_{i,s}&=(q^2w_{i,r}-w_{i,s})\chi_{i,s}\chi_{i,r},\\
(w_{j,r}-q^2w_{j,s})\chi_{j,r}\chi_{j,s}&=(q^2w_{j,r}-w_{j,s})\chi_{j,s}\chi_{j,r},\\
(w_{i,r}-q^{-1}w_{j,s'})\chi_{i,r}\chi_{j,s'}&=(q^{-1}w_{i,r}-w_{j,s'})\chi_{j,s'}\chi_{i,r}.
\end{align*}
\end{lem}
\begin{proof}
Follows by a direct calculation.
\end{proof}

Therefore, by the same type calculation as in \S\ref{sec:relBB2} using Lemma \ref{lem:qs}, we find that all $\chi_{i,r}\chi_{j,s'}$ (and $\chi_{j,s'}\chi_{i,r}$) with $r\ne s$ in the LHS of \eqref{rel:BB5} cancel. Hence we are left with terms involving $\chi_{i,r}\chi_{j,r'}$ and $\chi_{j,r'}\chi_{i,r}$ summed over $r$ from the LHS and the scalar series from the RHS. Note that due to \eqref{eq:ext} we have $\partial_{i,r}=\partial_{j,r'}^{-1}$ and hence these terms do not involve the difference operators.

Set
\begin{align*}
\Xi_i(u)&=\zeta_i\zeta_{j}\brown{\frac{uq^{\wp_i}-(uq^{\wp_i})^{-1}}{u-u^{-1}}}\frac{\Z_i(u)}{\W_i(qu)\W_{i}(q^{-1}u)}\prod_{k\leftrightarrow i}\W_k(u),\\
\Xi_j(u)&=\zeta_i\zeta_{j}\brown{\frac{uq^{\wp_j}-(uq^{\wp_j})^{-1}}{u-u^{-1}}}\frac{\Z_j(u)}{\W_j(qu)\W_{j}(q^{-1}u)}\prod_{k\leftrightarrow j}\W_k(u).	
\end{align*}
Then by \eqref{eq:invariant} we have $\Xi_i(u)=\Xi_j(u^{-1})$. We can write
\[
\frac{\delta(uv)(u-v)}{q^2-1}\big(\Xi_i(u)-\Xi_j(v)\big)
\]
as either
\[
\frac{\delta(uv)(u-u^{-1})}{q^2-1}\big(\Xi_i(u)^+-\Xi_i(u)^-\big)\quad\text{or} \quad
\frac{\delta(uv)(v-v^{-1})}{q^2-1}\big(\Xi_j(v)^+-\Xi_j(v)^-\big).
\]
We shall use the following simple observation:
\beq\label{helper90}
\delta(uv)\delta \Big(\frac{a}{u}\Big)\mathrm{Res}_{u=a}\frac{(u-u^{-1})\Xi_i(u)}{(q^2-1)u}=\delta(uv)\delta \Big(\frac{1}{av}\Big)\mathrm{Res}_{v=a^{-1}}\frac{(v-v^{-1})\Xi_j(v)}{(q^2-1)v},
\eeq
where $a$ is a simple pole of $\Xi_i(u)$, excluding $\pm 1$ due to the factor $(u-u^{-1})$.

Note that $i\to j$. Expressing $\chi_{i,r}\chi_{j,r'}$ for $1\lle r\lle \bv_i$ explicitly, we have
\be
\chi_{i,r}\chi_{j,r'}=-\brown{\frac{1-q^{-3}w_{i,r}^2}{1-q^{-1}w_{i,r}^2}}\frac{q^{\frac32}\zeta_i \zeta_{j}\Z_i(q^{-1}w_{i,r})}{(1-q^2)^2\W_{i,r}(w_{i,r})\W_{i,r}(q^{-2}w_{i,r})}\prod_{k\leftrightarrow i}\W_k(q^{-1}w_{i,r}).
\ee
Thus one can directly verify that for $1\lle r\lle \bv_i$,
\beq\label{helper91}
\begin{split}
(u-q^{-1}v)\delta\Big(\frac{w_{i,r}}{qu}\Big)\delta\Big(\frac{qw_{i,r}^{-1}}{v}\Big)\chi_{i,r}\chi_{j,r'}
=\delta(uv)\delta\Big(\frac{w_{i,r}}{qu}\Big)\mathrm{Res}_{u=q^{-1}w_{i,r}}\frac{(u-u^{-1})\Xi_i(u)}{(q^2-1)u}.
\end{split}
\eeq
Here we used the equality
\[
q^{\frac32}\big(u-q^{-1}v\big)\frac{1-q^{-3}w_{i,r}^2}{1-q^{-1}w_{i,r}^2}=\big(uq^{-\frac12}-u^{-1}q^{\hf}\big),
\]
when $u=q^{-1}w_{i,r}$ and $v=qw_{i,r}^{-1}$.

For $\bv_i< r\lle n_i$, we have
\[
\chi_{i,r}\chi_{j,r'}=-\brown{\frac{1-q^{3}w_{j,r'}^2}{1-qw_{j,r'}^2}}\frac{q^{\frac32}\zeta_i \zeta_{j}\Z_j(qw_{j,r'})}{(1-q^2)^2\W_{j,r'}(w_{j,r'})\W_{j,r'}(q^{2}w_{j,r'})}\prod_{k\leftrightarrow j}\W_k(qw_{j,r'}),
\]
and further that
\beq\label{helper92}
\begin{split}
(u-q^{-1}v)\delta\Big(\frac{w_{i,r}}{qu}\Big)\delta\Big(\frac{qw_{i,r}^{-1}}{v}\Big)\chi_{i,r}\chi_{j,r'}
&\,=\,\delta(uv)\delta\Big(\frac{qw_{j,r'}}{v}\Big)\mathrm{Res}_{v=qw_{j,r'}}\frac{(v-v^{-1})\Xi_j(v)}{(q^2-1)v}\\
&\stackrel{\eqref{helper90}}{=}\delta(uv)\delta\Big(\frac{w_{i,r}}{qu}\Big)\mathrm{Res}_{u=q^{-1}w_{i,r}}\frac{(u-u^{-1})\Xi_i(u)}{(q^2-1)u}.	
\end{split}
\eeq

For $1\lle r\lle \bv_i$, we have
\[
\chi_{j,r'}\chi_{i,r}=-\frac{q^{\brown{\frac52}}\zeta_i \zeta_{j}\Z_j(qw_{i,r})}{(1-q^2)^2\W_{i,r}(w_{i,r})\W_{i,r}(q^{2}w_{i,r})}\prod_{k\leftrightarrow i}\W_k(qw_{i,r}),
\]
where the exponent for $q$ is increased by 1 due to moving $\partial_{i,r}$ through $w_{i,r}^{1/2}$.
Therefore, we obtain
\begin{align}\label{helper93}
(v-q^{-1}u)\delta\Big(\frac{qw_{i,r}}{u}\Big)\delta\Big(\frac{w_{i,r}^{-1}}{qv}\Big)\chi_{j,r'}\chi_{i,r}=\delta(uv)\delta\Big(\frac{qw_{i,r}}{u}\Big)\mathrm{Res}_{u=qw_{i,r}}\frac{(u-u^{-1})\Xi_i(u)}{(q^2-1)u}.	
\end{align}

For $\bv_i< r\lle n_i$, we have
\[
\chi_{j,r'}\chi_{i,r}=-\frac{q^{\brown{\frac12}}\zeta_i \zeta_{j}\Z_j(q^{-1}w_{j,r'})}{(1-q^2)^2\W_{j,r'}(w_{j,r'})\W_{j,r'}(q^{-2}w_{j,r'})}\prod_{k\leftrightarrow j}\W_k(q^{-1}w_{j,r'}),
\]
where the exponent for $q$ is decreased by 1 due to moving $\partial_{j,r'}^{-1}$ through $w_{j,r'}^{1/2}$.
Hence we obtain
\beq\label{helper94}
\begin{split}
(v-q^{-1}u)\delta\Big(\frac{qw_{i,r}}{u}\Big)\delta\Big(\frac{w_{i,r}^{-1}}{qv}\Big)\chi_{j,r'}\chi_{i,r} 
&\,=\,\delta(uv)\delta\Big(\frac{w_{j,r'}}{qv}\Big)\mathrm{Res}_{v=q^{-1}w_{j,r'}}\frac{(v-v^{-1})\Xi_j(v)}{(q^2-1)v}\\
&\stackrel{\eqref{helper90}}{=}\,\delta(uv)\delta\Big(\frac{qw_{i,r}}{u}\Big)\mathrm{Res}_{u=qw_{i,r}}\frac{(u-u^{-1})\Xi_i(u)}{(q^2-1)u}.	
\end{split}
\eeq
Combining \eqref{helper91}--\eqref{helper94}, by Lemma \ref{exp}, the relation \eqref{rel:BB3} is preserved by the map $\Phi_\mu^\la$.

\subsubsection{The relation \eqref{rel:Serre3}}
Consider the image of the LHS of \eqref{rel:Serre3} under the map $\Phi_\mu^\la$,
\beq\label{hp11}
\mathrm{Sym}_{u_1,u_2}\sum_{r_1,r_2=1}^{n_i}\sum_{s=1}^{n_i}\bigg[\delta\Big(\frac{w_{i,r_1}}{qu_1}\Big)\chi_{i,r_1},\Big[\delta\Big(\frac{w_{i,r_2}}{qu_2}\Big)\chi_{i,r_2},\delta\Big(\frac{w_{j,s'}}{qv}\Big)\chi_{j,s'}\Big]_q\bigg]_{q^{-1}}.
\eeq
If $r_1\ne s$ and $r_2\ne s$, then it follows from a calculation similar to that in \S\ref{sec:serre-split} or \cite[C(vii)]{FT19} using Lemma \ref{lem:qs} that 
\[
\mathrm{Sym}_{u_1,u_2}\bigg[\delta\Big(\frac{w_{i,r_1}}{qu_1}\Big)\chi_{i,r_1},\Big[\delta\Big(\frac{w_{i,r_2}}{qu_2}\Big)\chi_{i,r_2},\delta\Big(\frac{w_{j,s'}}{qv}\Big)\chi_{j,s'}\Big]_q\bigg]_{q^{-1}}=0.
\]
Thus only the terms with either $r_1=s$ or $r_2=s$ in \eqref{hp11} survive and will make a nontrivial contribution to the LHS of \eqref{rel:Serre3} under the map $\Phi_\mu^\la$. Denote by $\mathscr X_r$ the sum of the surviving terms containing $\chi_{i,r}$ from \eqref{hp11}, for $1\lle r\lle n_i$. From the discussion above, we have 
\[
\eqref{hp11}=\sum_{r=1}^{n_i}\mathscr X_r,
\]
where
\[
\mathscr X_r=\mathscr X_r^\circ +\sum_{s=1,s\ne r}^{n_i}\big(\mathscr X_s'+\mathscr X_s''\big),
\]
with 
\begin{align}
\mathscr X_r^\circ=\mathrm{Sym}_{u_1,u_2}\bigg[\delta\Big(\frac{w_{i,r}}{qu_1}\Big)\chi_{i,r},\Big[\delta\Big(\frac{w_{i,r}}{qu_2}\Big)\chi_{i,r},\delta\Big(\frac{w_{j,r'}}{qv}\Big)\chi_{j,r'}\Big]_q\bigg]_{q^{-1}},\label{hp12}\\
\mathscr X_s'=\mathrm{Sym}_{u_1,u_2}\bigg[\delta\Big(\frac{w_{i,r}}{qu_1}\Big)\chi_{i,r},\Big[\delta\Big(\frac{w_{i,s}}{qu_2}\Big)\chi_{i,s},\delta\Big(\frac{w_{j,s'}}{qv}\Big)\chi_{j,s'}\Big]_q\bigg]_{q^{-1}},	\label{hp13}\\
\mathscr X_s''=\mathrm{Sym}_{u_1,u_2}\bigg[\delta\Big(\frac{w_{i,s}}{qu_1}\Big)\chi_{i,s},\Big[\delta\Big(\frac{w_{i,r}}{qu_2}\Big)\chi_{i,r},\delta\Big(\frac{w_{j,s'}}{qv}\Big)\chi_{j,s'}\Big]_q\bigg]_{q^{-1}}.\label{hp14}
\end{align}

As seen in \S\ref{ssec:BB3}, the power of $q$ may sometimes increase or decrease by 1 differently depending on $1\lle r\lle \bv_i$ or $\bv_i< r\lle n_i$. Below we shall only exhibit the calculations when $1\lle r\lle \bv_i$, as the other case can be handled similarly. 

We first consider the following from \eqref{hp12} (without symmetrization)
\[
\bigg[\delta\Big(\frac{w_{i,r}}{qu_1}\Big)\chi_{i,r},\Big[\delta\Big(\frac{w_{i,r}}{qu_2}\Big)\chi_{i,r},\delta\Big(\frac{w_{j,r'}}{qv}\Big)\chi_{j,r'}\Big]_q\bigg]_{q^{-1}},
\]
which expands as
\begin{align}
&\,\delta\Big(\frac{w_{i,r}}{qu_1}\Big)\delta\Big(\frac{w_{i,r}}{q^3u_2}\Big)\delta\Big(\frac{q^3w_{j,r'}}{v}\Big)\chi_{i,r}\chi_{i,r}\chi_{j,r'}\label{hp15}\\
+&\,\delta\Big(\frac{w_{i,r}}{qu_1}\Big)\delta\Big(\frac{qw_{i,r}}{u_2}\Big)\delta\Big(\frac{w_{j,r'}}{qv}\Big)\chi_{j,r'}\chi_{i,r}\chi_{i,r}\label{hp16}\\
-&\,(q+q^{-1})\delta\Big(\frac{w_{i,r}}{qu_1}\Big)\delta\Big(\frac{w_{i,r}}{qu_2}\Big)\delta\Big(\frac{qw_{j,r'}}{v}\Big)\chi_{i,r}\chi_{j,r'}\chi_{i,r}.\label{hp17}
\end{align}

Due to the delta functions, we can rewrite
\begin{align*}
\delta(u_2v)&\delta\Big(\frac{w_{i,r}}{qu_2}\Big)\delta\Big(\frac{w_{i,r}}{qu_1}\Big) \frac{(1+q^2)(u_2-v)v\Xi_i(u_2)}{(qu_1-v)(v-q^2u_1^{-1})u_2}\\
&=\delta(u_2v)\delta\Big(\frac{w_{i,r}}{qu_2}\Big)\delta\Big(\frac{w_{i,r}}{qu_1}\Big)\frac{(1+q^2)(u_2-u_2^{-1})u_2^{-1}\Xi_i(u_2)}{(w_{i,r}-u_2^{-1})(u_2^{-1}-q^3w_{i,r}^{-1})u_2}.
\end{align*}
Note that $\Xi_i(u_2)$ contains $(1-w_{i,r}u_2)$ from $\W_j(u_2)$ in the numerator which results in cancellation with the $(w_{i,r}-u_2^{-1})$ in the denominator of the function above. A slightly tedious calculation using \eqref{helper91}--\eqref{helper94} shows that
\begin{align*}
\eqref{hp15}&=\delta(u_2v)\delta\Big(\frac{w_{i,r}}{q^3u_2}\Big)\mathrm{Res}_{u_2=q^{-3}w_{i,r}}\frac{(1+q^2)(u_2-v)v\Xi_i(u_2)}{(qu_1-v)(v-q^2u_1^{-1})u_2}\delta\Big(\frac{w_{i,r}}{qu_1}\Big)\chi_{i,r},\\
\eqref{hp16}&=\delta(u_2v)\delta\Big(\frac{qw_{i,r}}{u_2}\Big)\mathrm{Res}_{u_2=qw_{i,r}}\frac{(1+q^2)(u_2-v)v\Xi_i(u_2)}{(qu_1-v)(v-q^2u_1^{-1})u_2}\delta\Big(\frac{w_{i,r}}{qu_1}\Big)\chi_{i,r},\\
\eqref{hp17}&=\delta(u_2v)\delta\Big(\frac{w_{i,r}}{qu_2}\Big)\mathrm{Res}_{u_2=q^{-1}w_{i,r}}\frac{(1+q^2)(u_2-v)v\Xi_i(u_2)}{(qu_1-v)(v-q^2u_1^{-1})u_2}\delta\Big(\frac{w_{i,r}}{qu_1}\Big)\chi_{i,r}.
\end{align*}

Now we expand the following term from \eqref{hp13},
\[
\bigg[\delta\Big(\frac{w_{i,r}}{qu_1}\Big)\chi_{i,r},\Big[\delta\Big(\frac{w_{i,s}}{qu_2}\Big)\chi_{i,s},\delta\Big(\frac{w_{j,s'}}{qv}\Big)\chi_{j,s'}\Big]_q\bigg]_{q^{-1}},
\]
as
\begin{align}
\delta\Big(\frac{w_{i,r}}{qu_1}\Big)\delta\Big(\frac{w_{i,s}}{qu_2}\Big)\delta\Big(\frac{qw_{j,s'}}{v}\Big)\chi_{i,r}\chi_{i,s}\chi_{j,s'}-q^{-1}\delta\Big(\frac{w_{i,r}}{qu_1}\Big)\delta\Big(\frac{w_{i,s}}{qu_2}\Big)\delta\Big(\frac{qw_{j,s'}}{v}\Big)\chi_{i,s}\chi_{j,s'}\chi_{i,r}\label{type5-1}\\
- q\delta\Big(\frac{w_{i,r}}{qu_1}\Big)\delta\Big(\frac{qw_{i,s}}{u_2}\Big)\delta\Big(\frac{w_{j,s'}}{qv}\Big)\chi_{i,r}\chi_{j,s'}\chi_{i,s}+\delta\Big(\frac{w_{i,r}}{qu_1}\Big)\delta\Big(\frac{qw_{i,s}}{u_2}\Big)\delta\Big(\frac{w_{j,s'}}{qv}\Big)\chi_{j,s'}\chi_{i,s}\chi_{i,r},\label{type3-1}
\end{align}
and the following term from \eqref{hp14},
\begin{align}
&\bigg[\delta\Big(\frac{w_{i,s}}{qu_2}\Big)\chi_{i,s},\Big[\delta\Big(\frac{w_{i,r}}{qu_1}\Big)\chi_{i,r},\delta\Big(\frac{w_{j,s'}}{qv}\Big)\chi_{j,s'}\Big]_q\bigg]_{q^{-1}}\notag\\
=\,&\bigg[\delta\Big(\frac{w_{i,s}}{qu_2}\Big)\chi_{i,s},\delta\Big(\frac{w_{i,r}}{qu_1}\Big)\delta\Big(\frac{w_{j,s'}}{qv}\Big) \frac{(q^{-1}-q)w_{i,r}}{q^{-1}w_{i,r}- w_{j,s'}}\chi_{i,r}\chi_{j,s'} \bigg]_{q^{-1}}\notag	\\
=\,&\delta\Big(\frac{w_{i,r}}{qu_1}\Big)\delta\Big(\frac{w_{i,s}}{qu_2}\Big)\delta\Big(\frac{qw_{j,s'}}{v}\Big)\frac{(q^{-1}-q)w_{i,r}}{q^{-1}w_{i,r}- q^2w_{j,s'}}\chi_{i,s}\chi_{i,r}\chi_{j,s'}\label{type5-2}\\
&-q^{-1}\delta\Big(\frac{w_{i,r}}{qu_1}\Big)\delta\Big(\frac{qw_{i,s}}{u_2}\Big)\delta\Big(\frac{w_{j,s'}}{qv}\Big)\frac{(q^{-1}-q)w_{i,r}}{q^{-1}w_{i,r}- w_{j,s'}}\chi_{i,r}\chi_{j,s'} \chi_{i,s}\label{type3-2}.
\end{align}
The sum of the first term in \eqref{type3-1} and \eqref{type3-2} is equal to
\beq\label{type3-3}
-\frac{q^{-2}w_{i,r}-qw_{j,s'}}{q^{-1}w_{i,r}-w_{j,s'}}\delta\Big(\frac{w_{i,r}}{qu_1}\Big)\delta\Big(\frac{qw_{i,s}}{u_2}\Big)\delta\Big(\frac{w_{j,s'}}{qv}\Big)\chi_{i,r}\chi_{j,s'} \chi_{i,s}.
\eeq
Using the equality
\[
\chi_{i,r}\chi_{j,s'}\chi_{i,s}=\frac{q^{-1}(1-q^{-1}w_{i,r}w_{j,s'}^{-1})(1-w_{i,r}w_{j,s'})}{(1-qw_{i,r}w_{j,s'}^{-1})(1-q^{-4}w_{i,r}w_{j,s'})}\chi_{j,s'}\chi_{i,s}\chi_{i,r}
\]
and
\begin{align*}
1-\frac{q^{-2}w_{i,r}-qw_{j,s'}}{q^{-1}w_{i,r}-w_{j,s'}}&\frac{q^{-1}(1-q^{-1}w_{i,r}w_{j,s'}^{-1})(1-w_{i,r}w_{j,s'})}{(1-qw_{i,r}w_{j,s'}^{-1})(1-q^{-4}w_{i,r}w_{j,s'})}\\
&=\frac{q^{-2}w_{i,r}(1-q^4)(1-q^{-1}w_{j,s'}^2)}{q(w_{j,s'}-qw_{i,r})(1-q^{-4}w_{i,r}w_{j,s'})},
\end{align*}
together with \eqref{helper91}--\eqref{helper94}, we find that 
\[
\eqref{type3-1}+\eqref{type3-2}=\delta(u_2v)\delta\Big(\frac{qw_{i,s}}{u_2}\Big)\mathrm{Res}_{u_2=qw_{i,s}}\frac{(1+q^2)(u_2-v)v\Xi_i(u_2)}{(qu_1-v)(v-q^2u_1^{-1})u_2}\delta\Big(\frac{w_{i,r}}{qu_1}\Big)\chi_{i,r}.
\]

By Lemma \ref{lem:qs}, we have
\[
\eqref{type5-2}=\delta\Big(\frac{w_{i,r}}{qu_1}\Big)\delta\Big(\frac{w_{i,s}}{qu_2}\Big)\delta\Big(\frac{qw_{j,s'}}{v}\Big)\frac{(q^{-1}-q)w_{i,r}}{w_{i,r}-qw_{j,s'}}\chi_{i,s}\chi_{j,s'}\chi_{i,r}.
\]
Hence the sum of the second term in \eqref{type5-1} and \eqref{type5-2} is equal to
\[
\delta\Big(\frac{w_{i,r}}{qu_1}\Big)\delta\Big(\frac{w_{i,s}}{qu_2}\Big)\delta\Big(\frac{qw_{j,s'}}{v}\Big)\frac{w_{j,s'}-qw_{i,r}}{w_{i,r}-qw_{j,s'}} \chi_{i,s}\chi_{j,s'}\chi_{i,r}.
\]
Using this and the equality
\[
\chi_{i,r}\chi_{i,s}\chi_{j,s'}=\frac{q^{-1}(1-q^{-3}w_{i,r}w_{j,s'}^{-1})(1-q^2w_{i,r}w_{j,s'})}{(1-q^{-2}w_{i,r}w_{j,s'})(1-q^{-1}w_{i,r}w_{j,s'}^{-1})}\chi_{i,s}\chi_{j,s'}\chi_{i,r}
\]
together with 
\eqref{helper91}--\eqref{helper94}, 
we obtain 
\begin{align*}
\eqref{type5-1}+\eqref{type5-2}&=\delta\Big(\frac{w_{i,r}}{qu_1}\Big)\delta\Big(\frac{w_{i,s}}{qu_2}\Big)\delta\Big(\frac{qw_{j,s'}}{v}\Big)
\frac{(q^4-1)w_{i,r}(1-qw_{j,s'}^2)}{q(q^2-w_{i,r}w_{j,s'})(qw_{j,s'}-w_{i,r})}\chi_{i,s}\chi_{j,s'}\chi_{i,r}
\\
&=\delta(u_2v)\delta\Big(\frac{w_{i,s}}{qu_2}\Big)\mathrm{Res}_{u_2=q^{-1}w_{i,s}}\frac{(1+q^2)(u_2-v)v\Xi_i(u_2)}{(qu_1-v)(v-q^2u_1^{-1})u_2}\delta\Big(\frac{w_{i,r}}{qu_1}\Big)\chi_{i,r}.
\end{align*}
Summing all these equations involving residues in this subsection, it follows from Lemma \ref{exp} that this sum is equal to the component $\delta\big(\frac{w_{i,r}}{qu_1}\big)\chi_{i,r}$ in the image of 
\[
\frac{\delta(u_2v)(1+q^2)(u_2-v)v}{(qu_1-v)(v-q^2u_1^{-1})}\big(\acute{\Theta}_i(u_2)-\acute{\Theta}_{\tau i}(v)\big)B_i(u_1)
\]
under the map $\Phi_{\mu}^{\la}$. Finally, taking the symmetrization on $u_1,u_2$, we complete the proof that the relation \eqref{rel:Serre3} is preserved by $\Phi_\mu^\la$.

\bibliographystyle{amsalpha}
\bibliography{references}

\end{document}